\documentclass[11pt]{article}
\usepackage{hyperref}
\usepackage{amsmath}
\usepackage{amssymb}
\usepackage{amscd}
\usepackage{textcomp}

\usepackage{amsthm}
\theoremstyle{plain}
\newtheorem{theorem}{Theorem}[section]

\newtheorem{lemma}[theorem]{Lemma}

\newtheorem{corollary}[theorem]{Corollary}
\newtheorem{proposition}[theorem]{Proposition}

\theoremstyle{definition}
\newtheorem{definition}[theorem]{Definition}

\newtheorem{example}[theorem]{Example}
\newtheorem{remark}[theorem]{Remark}

\theoremstyle{remark}

\usepackage{xy}
\xyoption{all}

\newdir^{ (}{{}*!/-3pt/\dir^{(}}    
\newdir^{  }{{}*!/-3pt/\dir^{}}    
\newdir_{ (}{{}*!/-3pt/\dir_{(}}    
\newdir_{  }{{}*!/-3pt/\dir_{}}

\newcounter{zahl}

\def\theenumi{(\alph{enumi})}

\def\p@enumii{\theenumi}

\newcommand{\DS}{\displaystyle}
\newcommand{\TS}{\textstyle}
\newcommand{\SC}{\scriptstyle}
\newcommand{\SSC}{\scriptscriptstyle}

\DeclareMathOperator{\Frob}{Frob}

\DeclareMathOperator{\GL}{GL}
\DeclareMathOperator{\GSp}{GSp}

\DeclareMathOperator{\Ind}{Ind}

\DeclareMathOperator{\Spec}{Spec}
\DeclareMathOperator{\Spf}{Spf}

\newcommand{\alg}{{\rm alg}}

\DeclareMathOperator{\codim}{codim}

\newcommand{\dom}{{\rm dom}}

\newcommand{\fpqc}{{\it fpqc}}
\newcommand{\fppf}{{\it fppf}}

\DeclareMathOperator{\rk}{rk}
\newcommand{\semidir}{\mbox{$\times\hspace{-1mm}\raisebox{0.17mm}{$\shortmid$}$}}

\renewcommand{\phi}{\varphi}
\renewcommand{\epsilon}{\varepsilon}
\usepackage{amsfonts}
\newcommand{\BOne} {{\mathchoice{\hbox{\rm1\kern-2.7pt l\kern.9pt}}
                              {\hbox{\rm1\kern-2.7pt l\kern.9pt}}
                              {\hbox{\scriptsize\rm1\kern-2.3pt l\kern.4pt}}
                              {\hbox{\scriptsize\rm1\kern-2.4pt l\kern.5pt}}}}

\newcommand{\BF}{{\mathbb{F}}}
\newcommand{\BG}{{\mathbb{G}}}

\newcommand{\BN}{{\mathbb{N}}}

\newcommand{\BQ}{{\mathbb{Q}}}

\newcommand{\BZ}{{\mathbb{Z}}}

\newcommand{\CB}{{\cal{B}}}
\newcommand{\CC}{{\cal{C}}}
\newcommand{\CalD}{{\cal{D}}}

\newcommand{\CF}{{\cal{F}}}
\newcommand{\CG}{{\cal{G}}}
\newcommand{\CH}{{\cal{H}}}
\newcommand{\CI}{{\cal{I}}}

\newcommand{\CL}{{\cal{L}}}

\newcommand{\CN}{{\cal{N}}}
\newcommand{\CO}{{\cal{O}}}
\newcommand{\CP}{{\cal{P}}}

\newcommand{\FP}{{\mathfrak{P}}}

\newcommand{\Fb}{{\mathfrak{b}}}

\newcommand{\Fm}{{\mathfrak{m}}}

\newcommand{\Fp}{{\mathfrak{p}}}

\newcommand{\rsurj}{\mbox{\mathsurround=0pt \;$\longrightarrow \hspace{-0.7em} \to$\;}}

\newcommand{\es}{\enspace}

\newcommand{\dual}{^{\SSC\lor}}

\newcommand{\ul}[1]{{\underline{#1}}}
\newcommand{\ol}[1]{{\overline{#1}}}
\newcommand{\wh}[1]{{\widehat{#1}}}
\newcommand{\wt}[1]{{\widetilde{#1}}}
\newcommand{\whtimes}{\wh{\raisebox{0ex}[0ex]{$\times$}}}

\usepackage{ifthen}

\newcommand{\invlim}[1][]{\ifthenelse{\equal{#1}{}}% falls Argument leer
{\DS \lim_{\longleftarrow}}%                         verwende niedrige Version
{\DS \lim_{\underset{#1}{\longleftarrow}}}%  sonst:  verwende Argument
}

\newcommand{\dirlim}[1][]{\ifthenelse{\equal{#1}{}}% falls Argument leer
{\DS \lim_{\longrightarrow}}%                        verwende niedrige Version
{\DS \lim_{\underset{#1}{\longrightarrow}}}% sonst:  verwende Argument
}

\newcommand{\dbl}{{\mathchoice{\mbox{\rm [\hspace{-0.15em}[}}
                              {\mbox{\rm [\hspace{-0.15em}[}}
                              {\mbox{\scriptsize\rm [\hspace{-0.15em}[}}
                              {\mbox{\tiny\rm [\hspace{-0.15em}[}}}}
\newcommand{\dbr}{{\mathchoice{\mbox{\rm ]\hspace{-0.15em}]}}
                              {\mbox{\rm ]\hspace{-0.15em}]}}
                              {\mbox{\scriptsize\rm ]\hspace{-0.15em}]}}
                              {\mbox{\tiny\rm ]\hspace{-0.15em}]}}}}
\newcommand{\dpl}{{\mathchoice{\mbox{\rm (\hspace{-0.15em}(}}
                              {\mbox{\rm (\hspace{-0.15em}(}}
                              {\mbox{\scriptsize\rm (\hspace{-0.15em}(}}
                              {\mbox{\tiny\rm (\hspace{-0.15em}(}}}}
\newcommand{\dpr}{{\mathchoice{\mbox{\rm )\hspace{-0.15em})}}
                              {\mbox{\rm )\hspace{-0.15em})}}
                              {\mbox{\scriptsize\rm )\hspace{-0.15em})}}
                              {\mbox{\tiny\rm )\hspace{-0.15em})}}}}

\DeclareMathOperator{\QIsog}{QIsog}

\newcommand{\Gr}{Gr}
\newcommand{\Flag}{\CF\!lag}

\newcommand{\T}{T}
\newcommand{\el}{l}

\def\s{\sigma^\ast}

\def\onto{\rsurj}
\def\isoto{\stackrel{}{\mbox{\hspace{1mm}\raisebox{+1.4mm}{$\SC\sim$}\hspace{-3.5mm}$\longrightarrow$}}}

\newbox\mybox
\def\arrover#1{\mathrel{
       \setbox\mybox=\hbox spread 1.4em{\hfil$\scriptstyle#1$\hfil}
       \vbox{\offinterlineskip\copy\mybox
             \hbox to\wd\mybox{\rightarrowfill}}}}

\newcommand{\Test}{Y}
\newcommand{\kkk}{\tilde f}
\newcommand{\rr}{r}
\def\ulV{{\underline{V\!}\,}{}}
\def\onto{\mbox{$\kern2pt\to\kern-8pt\to\kern2pt$}}

\begin{document}

%%%%%%%%%%%%%%%%%%%%%%%%%%%%%%%%%%%%%%%%%%%%%%%%%%%%%%%%%%%%%%%%%%%%%%

\author{Urs Hartl and Eva Viehmann}

\title{Foliations in deformation spaces of local $G$-shtukas}
\date{}

\maketitle

\begin{abstract}
We study local $G$-shtukas with level structure over a base scheme whose Newton polygons are constant on the base. We show that after a finite base change and after passing to an \'etale covering, such a local $G$-shtuka is isogenous to a completely slope divisible one, generalizing corresponding results for $p$-divisible groups by Oort and Zink. As an application we establish a product structure up to finite surjective morphism on the closed Newton stratum of the universal deformation of a local $G$-shtuka, similarly to Oort's foliations for $p$-divisible groups and abelian varieties. This also yields bounds on the dimensions of affine Deligne-Lusztig varieties and proves equidimensionality of affine Deligne-Lusztig varieties in the affine Grassmannian.

\end{abstract}

\bigskip

%%%%%%%%%%%%%%%%%%%%%%%%%%%%%%%%%%%%%%%%%%%%%%%%%%%%%%%%%%%%%%%%%%%%%%
%
%    Introduction
%
%%%%%%%%%%%%%%%%%%%%%%%%%%%%%%%%%%%%%%%%%%%%%%%%%%%%%%%%%%%%%%%%%%%%%%

\section{Introduction}
%\addcontentsline{toc}{section}{Introduction}

We denote by $\BF_q$ the finite field with $q$ elements and by $k$ an algebraically closed field which is also an $\BF_q$-algebra. 

Let $G$ be a split connected reductive group over $\BF_q$. Local $G$-shtukas are analogs over the function field $\BF_q\dpl z\dpr$ of $p$-divisible groups with extra structure. To define them, let $LG$ be the loop group of $G$ see \cite[Definition 1]{Faltings03}. I.e.~$LG$ is the ind-scheme over $\BF_q$ representing the sheaf of groups for the \fpqc-topology whose sections for an $\BF_q$-algebra $R$ are given by $LG(\Spec R)\,=\,G\bigl(R\dbl z\dbr[\frac{1}{z}]\bigr)$. We also use the analogous notion for example for parabolic subgroups of $G$. Let $K_0$ be the infinite dimensional affine group scheme over $\BF_q$ with $K_0(\Spec R)\,=\,G\bigl(R\dbl z\dbr\bigr)$. In the natural way $K_0$ can be viewed as a subsheaf of $LG$.  Every $K_0$-torsor over an $\BF_q$-scheme $S$ for the \fpqc-topology is already a $K_0$-torsor for the \'etale topology by \cite[Proposition~2.1]{HV1}. For such a $K_0$-torsor $\CG$ on $S$ let $\CL\CG$ be the associated $LG$-torsor and $\s\CL\CG$ the pullback of $\CL\CG$ under the $q$-Frobenius morphism $\Frob_q:S\to S$. Then a \emph{local $G$-shtuka} over $S$ is a pair $\ul\CG=(\CG,\phi)$ consisting of a $K_0$-torsor $\CG$ on $S$ and an isomorphism of $LG$-torsors $\phi:\s\CL\CG\isoto\CL\CG$. Local $G$-shtukas over $k$ can also be described as follows. There exists a trivialization $\CG\cong (K_0)_k$ and with respect to such a trivialization $\phi$ corresponds to an element $b\in LG(k)$. A change of the trivialization replaces $b$ by $g^{-1}b\s(g)$ for $g\in K_0(k)$, where $\s$ is the endomorphism of $K_0$ and of $LG$ induced by $\Frob_q$. The special case of local $\GL_n$-shtukas is a direct analog of $p$-divisible groups in this context of function fields \cite{HartlDict}. For $G=\GSp_{2n}$ one obtains the analog of polarized $p$-divisible groups. 
For a definition of local $G$-shtukas and their moduli over more general schemes see \cite{HV1}. In contrast to \cite{HV1} we here only consider local $G$-shtukas over  base schemes $S$ over $\BF_q=\BF_q\dbl\zeta\dbr/(\zeta)$ instead of over $\BF_q\dbl\zeta\dbr$. The restriction to base schemes with $\zeta=0$ corresponds to $p$-divisible groups over $\BF_p$-schemes instead of over $\BZ_p$-schemes. Whenever it does not require additional effort we also consider more generally local $G$-shtukas with a level structure by certain subgroups $K$ of $K_0$, the most interesting cases being $K=K_0$ (in which case one obtains the definition above) or $K=I$, an Iwahori subgroup of $K_0$. For simplicity we only consider the case $K=K_0$ in the introduction.

Let $B\supset T$ be a Borel subgroup and a maximal split torus of $G$. Recall that a coweight $\lambda\in X_\ast(\T)$ is called dominant if $\langle\lambda,\alpha\rangle\ge0$ for every positive root $\alpha$ of $G$. To a local $G$-shtuka $\ul\CG$ over an algebraically closed field $k$ we assign its Hodge point or Hodge polygon, i.e.~the unique dominant coweight $\mu=\mu_{\ul\CG}\in X_*(\T)$ such that the element $b$ associated with $\ul\CG$ is in the double coset $K_0(k)z^{\mu}K_0(k)$ where $z^{\mu}$ denotes the image of $z$ under $\mu:\mathbb{G}_m\rightarrow \T$. For $\ul\CG$ over a scheme $S$ we consider the function $\mu_{\ul\CG}$ assigning to each geometric point $s$ of $S$ the Hodge point $\mu_{\ul\CG}(s)$ of $\ul\CG_s$. It has the property that its image $\overline{\mu}_{\ul\CG}$ in $\pi_1(G)=X_*(\T)/({\rm coroot~lattice})$ is locally constant on $S$ (see \cite[Proposition 3.4]{HV1}). If $S$ is reduced we say that $\ul\CG$ is bounded by some dominant $\mu\in X_*(\T)$ if $\mu_{\ul\CG}(s)\preceq\mu$ for all $s\in S$. Here $\preceq$ is the Bruhat ordering on the set of dominant coweights of $G$, i.e.~$\mu_1\preceq\mu_2$ if and only if the difference $\mu_2-\mu_1$ is a non-negative integral linear combination of simple coroots. There is also a notion of boundedness for local $G$-shtukas over non-reduced schemes, compare Definition \ref{DefBounded} or \cite[Definition 3.5]{HV1}. 

A second important invariant of a local $G$-shtuka over an algebraically closed field is its isogeny class. It corresponds to the $\sigma$-conjugacy class $[b]=\{g^{-1}b\s(g):g\in LG(k)\}$ of the associated element $b\in LG(k)$. In \cite{Kottwitz85}, Kottwitz classifies the set $B(G)$ of $\sigma$-conjugacy classes of elements $b\in LG(k)$ by two invariants. One of them is $\kappa(b):=\overline{\mu}_{(K_k,b\s)}$. Note that it can also be defined directly from $b$ via the Kottwitz map $\kappa:LG(k)\rightarrow \pi_1(G)$. Especially, it only depends on $[b]$. The other invariant is the \emph{Newton point} or \emph{Newton polygon} $\nu_b\in X_\ast(\T)_{\mathbb{Q}}$ of $b$. For $G=\GL_n$ this is the usual Newton polygon of the $\sigma$-linear map $b\s$, for general $G$ it is defined by requiring that it be functorial in $G$. The two invariants have the same image in $\pi_1(G)_{\mathbb{Q}}$. For $\ul\CG=(K_k,b\s)$ we have $\nu_{\ul\CG}\preceq\mu_{\ul\CG}$ as elements of $X_\ast(\T)_{\mathbb{Q}}$ (see \cite[Proposition 7.2]{HV1}).
The Newton point of a local $G$-shtuka $\ul\CG$ over a scheme $S$ is defined as the function assigning to each $k$-valued point $s$ of $S$ the Newton point $\nu=\nu_{\ul\CG}(s)$ of $\ul\CG_s$. Over schemes, and already over general fields, two local $G$-shtukas whose Newton points coincide are not necessarily isogenous. Let $(\nu,\overline{\mu})$ be the pair of invariants associated with $[b]$ for some $b\in LG(k)$. Let $$\mathcal{N}_{\nu}=\{s\in S(k): \nu_{\ul\CG}(s)=\nu, \overline{\mu}_{\ul\CG_s}=\overline{\mu}\}.$$ By \cite[Theorem 7.3]{HV1} (see also \cite[Theorem 3.6]{RapoportRichartz}) this is the set of points of a locally closed subscheme of $S$ which is open in the analogously defined subscheme where we replace equality of Newton points by the condition $\nu_{\ul\CG}(s)\preceq\nu$. By abuse of notation we do not include $\overline{\mu}$ in the notation (although the stratum depends on $\overline{\mu}$) and call $\mathcal{N}_{\nu}$ the Newton stratum associated with $\nu$. The reason for this is that $\overline{\mu}$ is locally constant on $S$, so the additional datum $\overline{\mu}$ only singles out specific connected components. 

We now describe the results obtained in this article. Let $\ul\CG$ be a local $G$-shtuka over a noetherian scheme $S$ whose Newton point is constant. Let $M$ be the centralizer of the Newton point. It is the Levi component containing $T$ of a standard parabolic subgroup $P$ of $G$, i.e.~of a parabolic subgroup containing $B$. Let $N$ be the unipotent radical of $P$. Let $\ol P$ be the opposite parabolic. We say that $\ul\CG$ has a slope filtration if there exists a local $\ol P$-shtuka $\ul{\bar\CP}=(\bar\CP,\phi)$ over $S$ such that $\ul\CG\cong i_\ast\ul{\bar\CP}$ where $i:\ol P\hookrightarrow G$ is the inclusion morphism, and \'etale locally on $S$ the Frobenius $\phi$ is represented by an element $b\in L\ol P(S)$ with decomposition $b=m\ol n$ for some $m$ with $M$-dominant Newton point equal to $\nu$.

In Section \ref{seccsd} we consider a stronger condition than having a slope filtration, namely to be completely slope divisible. In addition to having a slope filtration it requires certain integrality conditions on the elements $m$ and $\ol n$. It is related to the condition for $p$-divisible groups to be completely slope divisible defined by Zink in \cite[Definition 10]{Zink}, compare Section \ref{secgln}. Our main result on slope filtrations or rather on completely slope divisible local $G$-shtukas is the following theorem.\\

\noindent{\bf Theorem \ref{ThmOZ2.4}.} {\it Let $S$ be an integral noetherian $\BF_q$-scheme and let $g\in LG(S)$ be such that the $\sigma$-conjugacy classes of $g_s$ in the geometric points $s$ of  $S$ all coincide. Let $x$ be a $P$-fundamental alcove associated with this $\sigma$-conjugacy class as in Remark \ref{remfundalc} \ref{a}. Then there are morphisms $\wt S\xrightarrow{\beta} S'\xrightarrow{\alpha} S$  with $S'$ integral, $\alpha$ finite surjective, and $\beta$ an \'etale covering, and a bounded element $\tilde h\in LG(\wt S)$ satisfying $\tilde h^{-1}g_{\wt S}\s(\tilde h)\in I(\wt S)xI(\wt S)$.}\\

Here $I$ denotes the standard Iwahori subgroup scheme of $K_0$. This theorem together with Proposition \ref{PropSlopeFilt} and Corollary \ref{CorCSD} implies that if $\ul\CG$ is a local $G$-shtuka over a scheme $S$ whose quasi-isogeny class is constant in all geometric points of $S$, then after a suitable finite base change and after passing to an \'etale covering $\ul\CG$ is isogenous to a local $G$-shtuka which is completely slope divisible. This is an analog of \cite[Theorem 7]{Zink} and \cite[Lemma 2.4]{OortZink}. The proof is also similar to the one given by Zink respectively Oort and Zink, however there is one ingredient \cite[Lemma 9]{Zink} which does not have an immediate generalization from $\GL_n$ to general reductive groups, so we replace it by a different argument here.

We then use this result to introduce and study the analogs of Oort's \cite{Oort04} foliations and central leaves in our situation. A description of central leaves for a $p$-divisible group whose Newton polygon has just two slopes which uses slope filtrations can be found in \cite{Chai}. Let $\BF/\BF_q$ be a field extension, $S$ a noetherian $\BF$-scheme, and let $\ul\CG$ be a bounded local $G$-shtuka over $S$. Let $\ul\BG$ be a local $G$-shtuka over $\BF$. We consider the subset of $S$ given by
\begin{eqnarray*}
\CC_{\ul\BG,S}&:=&\bigl\{\,s\in S:\ul\BG_k\cong\ul\CG_s\otimes_{k(s)}k\quad\text{over an algebraically closed extension }k/k(s)\,\bigr\}
\end{eqnarray*}
Here and throughout the article $k(s)$ denotes the residue field of $s$. In Corollary~\ref{ThmCentralLeaf} we show that this defines a reduced locally closed subscheme which is closed in the Newton stratum of $\ul\BG$. Any geometrically irreducible component of $\CC_{\ul\BG,S}$ with the induced reduced subscheme structure is called a \emph{central leaf corresponding to $\ul\BG$ in $S$}. 

In this way we do not only obtain a decomposition of each Newton stratum into a disjoint union of (in general infinitely many) closed subschemes, but in the special case of the Newton stratum of the universal deformation of a local $G$-shtuka bounded by some $\mu$ the leaves can be arranged in a continuous family. To formulate the corresponding result recall the following definition. The closed affine Deligne-Lusztig variety associated with $b\in LG(k)$ and a dominant $\mu\in X_*(\T)$ is the reduced  closed subscheme $X_{\preceq\mu}(b)$ of the affine Grassmannian $LG/K_0$ with $$X_{\preceq\mu}(b)(k)=\{g\in LG(k)/K_0(k):g^{-1}b\s(g)\in K_0(k)z^{\mu'}K_0(k)\text{ for some }\mu'\preceq\mu\}.$$ It is the underlying topological space of a Rapoport-Zink space for local $G$-shtukas; see~\cite[Theorem 6.3]{HV1}. There are also variants where $K_0$ is replaced by an Iwahori subgroup or where we require $\mu'=\mu$ (see Section \ref{secadlv}). In Theorem \ref{ThmProductStructure} we prove (a slight generalization of) the following theorem.\\

\noindent{\bf Theorem.} {\it Let $\ul{\BG}=\bigl(K_{0,k},b\s\bigr)$ be a local $G$-shtuka over an algebraically closed field $k$ bounded by a dominant $\mu\in X_*(\T)$ and with Newton point $\nu$. Let $\CN_\nu$ be the Newton stratum of $[b]$ in the universal deformation space of $\ul{\BG}$ bounded by $\mu$. Let $x$ be a $P$-fundamental alcove associated with $[b]$. Then there is a reduced scheme $S$ and a finite surjective morphism $S\twoheadrightarrow\CN_{\nu}$ which factorizes into finite surjective morphisms $S\rightarrow X_{\preceq\mu}(b)^\wedge\,\whtimes_k\, \CI^{\wedge}\rightarrow \mathcal{N}_{\nu}$. Here $X_{\preceq\mu}(b)^\wedge$ denotes the completion of $X_{\preceq\mu}(b)$ at $1$ and $\CI^{\wedge}$ denotes the completion of $\mathbb{A}^{\langle 2\rho,\nu\rangle}$ at $0$. Furthermore, $\CC_{\ul{\BG},\mathcal{N}_{\nu}}$ is geometrically irreducible and equal to the image of $\{1\}\whtimes_k\CI^{\wedge}$ in $\CN_{\nu}$.}\\

\noindent In other words, the Newton stratum is up to a finite surjective morphism a product of a closed affine Deligne-Lusztig variety and a central leaf which is an affine space in this case. The dimension of each irreducible component of the Newton stratum of the universal deformation of a local $G$-shtuka bounded by some $\mu$ can be bounded below using the purity of the Newton stratification \cite[Theorem 7.4]{HV1}. Using this and the fact that the dimension of the central leaves is also known we obtain lower bounds on the dimension of each irreducible component of the affine Deligne-Lusztig varieties. For the affine Deligne-Lusztig varieties in the affine Grassmannian discussed above this implies that they are equidimensional of some dimension which is known by \cite{GHKR},\cite{Viehmann06}.\\

\noindent{\bf Corollary \ref{appladlv}(a).} {\it Let $b\in LG(k)$ and let $\nu$ be the Newton point of $b$. Let $\mu\in X_*(\T)$ be dominant with $X_{\mu}(b)\neq \emptyset$. Then $X_{\preceq\mu}(b)$ and $X_{\mu}(b)$ are equidimensional of dimension $\langle\rho,\mu-\nu\rangle-\frac{1}{2}(\rk_{\mathbb{F}_q\dpl z\dpr}(G)-\rk_{\mathbb{F}_q\dpl z\dpr}(J_b))$. In particular, $X_{\mu}(b)$ is open and dense in $X_{\preceq\mu}(b)$.}\\

\noindent In \cite{HV1} we proved the corollary under the additional assumption that $b$ is basic. Equidimensionality of $X_{\mu}(b)$ for $b\in\T(\mathbb{F}_q\dpl z\dpr)$ (not necessarily basic) is shown in \cite[Proposition 2.17.1]{GHKR}. For affine Deligne-Lusztig varieties in the affine flag variety the bounds one obtains in this way are in general strictly smaller than the dimension of the variety. In this case there are examples where the affine Deligne-Lusztig varieties are indeed not equidimensional, see \cite{GoertzHe}. The results about affine Deligne-Lusztig varieties in the affine flag variety prove a conjecture of Beazley \cite[Conjecture 1]{Beazley}.

\bigskip

\noindent{\it Acknowledgments.} The first author acknowledges support of the DFG (German Research Foundation) in form of DFG-grant HA3006/2-1 and SFB 478, project C5. This work has also been partially supported by SFB/TR 45 ``Periods, Moduli Spaces and Arithmetic of Algebraic Varieties'' of the DFG.

\section{Quotients of loop groups and lifting properties}

Let $\CO_S\dbl z\dbr$ be the sheaf of $\CO_S$-algebras on $S$ for the \fpqc-topology whose ring of sections on an $S$-scheme $\Test$ is the ring of power series $\CO_S\dbl z\dbr(\Test):=\Gamma(\Test,\CO_\Test)\dbl z\dbr$.  This is indeed a sheaf being the countable direct product of $\CO_S$. Let $\CO_S\dpl z\dpr$ be the \fpqc-sheaf of $\CO_S$-algebras on $S$ associated with the presheaf $\Test\mapsto\Gamma(\Test,\CO_\Test)\dbl z\dbr[\frac{1}{z}]$. If $\Test$ is quasi-compact then $\CO_S\dpl z\dpr(\Test)=\Gamma(\Test,\CO_\Test)\dbl z\dbr[\frac{1}{z}]$ by \cite[Exercise II.1.11]{Hartshorne}. 

We consider the ordering $\preceq$ on the group of coweights $X_\ast(T)$ of $G$, which is defined as $\mu_1\preceq\mu_2$ if and only if the difference $\mu_2-\mu_1$ is a non-negative integral linear combination of simple coroots. If $\mu_1,\mu_2$ are dominant it coincides with the Bruhat ordering. On $X^*(T)$ we consider the analogous ordering. On $X_*(T)_{\mathbb{Q}}=X_*(T)\otimes_{\mathbb{Z}}\mathbb{Q}$ we use the ordering with $\mu_1\preceq\mu_2$ if and only if the difference $\mu_2-\mu_1$ is a non-negative rational linear combination of simple coroots. For $\mu\in X_*(T)$ we denote by $z^{\mu}$ the image of $z\in\mathbb{G}_m\bigl(\mathbb{F}_q\dpl z\dpr\bigr)$ under the morphism $\mu:\mathbb{G}_m\rightarrow T$.

For $n\geq 0$ let $K_{n}, I_n\subset K_0$ be the subgroup schemes defined by
\begin{align*}
K_{n}(S)&=\bigl\{\,g\in K_0(S): \es g\equiv 1\pmod{z^{n}}\bigr\}\\
I_n(S) &= \bigl\{\,g\in K_n(S): ( g~{\rm mod~}{z^{n+1}})\in B\bigl(\Gamma(S,\CO_S)\dbl z\dbr/(z^{n+1})\bigr)\bigr\}.
\end{align*} 
We write $I$ for the \emph{Iwahori subgroup scheme} $I_0$. We use the letter $K$ for arbitrary subgroup schemes of $K_0$ containing some $K_n$, not only for $K_0$ itself. 

Let $LG/K$ be the sheaf of sets associated with the presheaf $S\mapsto LG(S)/K(S)$ on the $\fppf$-site of $\BF_q$-schemes $S$. In particular $\Gr:=LG/K_0$ is called the \emph{affine Grassmannian} and $\Flag:=LG/I$ is called the \emph{affine flag variety}; see \cite[\S4.5]{BeilinsonDrinfeld}, \cite{BeauvilleLaszlo}, \cite{LaszloSorger}, who comprehensively develop the theory of the affine Grassmannian over the field of complex numbers. Most of their results, and in particular all we use here, also hold with the same proofs over $\BF_q$. For instance \cite{Ngo-Polo}, \cite{Faltings03} reprove some statements. Moreover \cite{PR2} present a generalization of results and proofs to twisted affine flag varieties over $\BF_q$. The sheaf $LG/K$ is represented by an ind-algebraic space over $\BF_q$ of ind-finite type. $\Gr$ and $\Flag$ are even ind-schemes. The latter is proved in the literature we just cited by embedding $LG/K$ as a closed ind-scheme of the corresponding ind-scheme for $G=\GL_\rr$. For $K=K_n$ and $K=I_n$ the same can be done by imposing level structures. In the general case one chooses $n$ such that $K_n\subset K$. Then $LG/K$ is the quotient of $LG/K_n$ by the smooth group scheme of finite presentation $K/K_n$. Using \cite[Exemple 4.6.1 and Corollaire 8.1.1]{LM} one concludes that $LG/K$ is an ind-algebraic space.

We consider right $K$-torsors and $LG$-torsors for the \'etale topology on $S$; see \cite[\S2]{HV1} for a thorough discussion of this notion.

\begin{lemma}\label{lemsecunip}
Let $G$ be a pro-unipotent group and $H$ a connected subgroup such that $G/H$ is finite-dimensional. Then the quotient map $G\rightarrow G/H$ has a section.
\end{lemma}

\begin{proof}
Note that we do not require $H$ to be normal. As $G$ is pro-unipotent we can choose a finite sequence of subgroups $G_0=G\supsetneq G_1\supsetneq \dotsm\supsetneq G_n=H$ with $G_i/G_{i+1}\cong \mathbb{G}_a$ for all $i$. We see that $G/G_i\rightarrow G/G_{i-1}$ is a $\mathbb{G}_a\cong \mathbb{A}^1$-bundle over $G/G_{i-1}$. Using induction and that line bundles over an affine space are trivial we obtain $G/G_{i}\cong \mathbb{A}^i$. As $H$ is pro-unipotent we can choose a sequence of subgroups $H_0=H\supsetneq H_1\supsetneq \dotsm$ with $\bigcap_{i\geq 0}H_i=(1)$ and $H_i/H_{i+1}\cong \mathbb{G}_a$. Using again induction and also a limit argument we see that it is enough to show that we have a section $G/H\rightarrow G/ H_1$. This follows again from the fact that line bundles over $\mathbb{A}^n$ are trivial.
\end{proof}

\begin{corollary}\label{corseckn}
\begin{enumerate}
\item For any $n> 0$, the projection $K_0\rightarrow K_0/K_n$ has a section. 
\item Let $I'$ be a connected subgroup of $I$ containing some $K_n$. Then $I\rightarrow I/I'$ has a section.
\end{enumerate}
\end{corollary}

\begin{proof}
For $n=1$ this is given by the inclusion $K_0/K_1(S)\cong G(S)\rightarrow G\bigl(\Gamma(S,\CO_S)\dbl z\dbr\bigr)\cong K_0(S)$. Using the short exact sequence $$0\rightarrow K_1/K_n\rightarrow K_0/K_n\rightarrow K_0/K_1\rightarrow 0$$ and the section for $n=1$ we see that it is enough to prove existence of a section for $K_1\rightarrow K_1/K_n$. That follows from the previous lemma. The second part is shown similarly.
\end{proof}

\begin{lemma}\label{LemmaTildeKEtale}
Let $K\subset K_0$ be a subgroup scheme containing some $K_n$. Let $S$ be an $\BF_q$-scheme and let $h:S\to LG/K$ be a morphism. Then over an \'etale covering $\wt S\to S$ the morphism $h$ lifts to an element $\tilde h\in LG(\wt S)$.
\end{lemma}

\begin{proof}
Let $n\in\mathbb{N}$ with $K_n\subset K$. Then $LG/K_n$ is a $(K_0/K_n)\times_{\BF_q}\Gr$-torsor over the affine Grassmannian $\Gr$ and a $( K/K_n)\times_{\BF_q} (LG/K)$-torsor over $LG/ K$. In particular, the projection $LG/K_n\to LG/K$ is smooth and the map $h$ lifts over an \'etale covering $S'\to S$ to a morphism $h':S'\to LG/K_n$.
$$\xymatrix{
\wt S\ar[r]\ar[d]^{\tilde h}&S'\ar[r]\ar[d]^{h'}&S\ar[d]^{h}\ar[dr]\\
LG\ar[r]&LG/K_n\ar[r]&LG/ K\ar[r]&\Gr
}$$

Consider the composite morphism $S'\to LG/K_n\to\Gr=LG/K_0$. By \cite[Theorem 4.5.1(iii)]{BeilinsonDrinfeld} this morphism lifts to an element $\tilde h\in LG(\wt S)$ over another \'etale covering $\wt S\to S'$. (Note that in \cite{BeilinsonDrinfeld} this is proved to hold even Zariski locally over an algebraically closed field of characteristic zero. Using \cite[Lemma 2.1]{Ngo-Polo} the proof carries over for our base field $\BF_q$ after allowing a finite separable extension of $\BF_q$.) The two elements $h'_{\wt S}$ and $\tilde h$ of $(LG/K_n)(\wt S)$ have the same image in $\Gr$. Since 
\[
(LG/K_n)\times_{\BF_q}(K_0/K_n)\;\isoto\; (LG/K_n)\times_{\Gr}(LG/K_n)\,,\quad (h,g)\;\mapsto\;(h,hg)
\]
is an isomorphism, there is an element $\tilde g\in(K_0/K_n)(\wt S)$ with $\tilde h\tilde g=h'_{\wt S}$ in $(LG/K_n)(\wt S)$. Corollary~\ref{corseckn}(a) implies that $\tilde g$ lifts to an element of $K_0(\wt S)$. Replacing $\tilde h$ by $\tilde h\tilde g$ finishes the proof.
\end{proof}

Let $\wt W=N_{\T}\bigl(k\dpl z\dpr\bigr)/\T\bigl(k\dbl z\dbr\bigr)$ be the extended affine Weyl group of $G$. Here $N_{\T}$ is the normalizer of $\T$ in $G$. We have $\wt W=X_*(\T)\semidir W$, where $W$ is the finite Weyl group of $G$. Another decomposition of $\wt W$ is $\wt W_0\semidir \Omega$ where $\Omega\cong\pi_1(G)$ consists of the elements $x\in\wt W$ with $xIx^{-1}=I$ and where $\wt W_0$ is a Coxeter group, generated by the (finite and infinite) simple reflections $s_i$ associated with our choice of $I$. A reference for properties of $\wt W$ that are less well known is \cite{AnhangPR}. By the affine Bruhat decomposition we have $LG(k)=\coprod_{x\in\wt W}I(k)xI(k)$. The closure relations between the double cosets are given by the Bruhat ordering on $\wt W$. I.e.~the closure of a double coset $I(k)xI(k)$ is the union of the double cosets for $y\leq x$. Here $y\leq x$ if and only if there is a reduced expression $x=s_{i_1}\dotsm s_{i_n}\tau$ with $\tau\in\Omega$ such that $y=s_{j_1}\dotsm s_{j_m}\tau$ for the same $\tau$ and such that $(j_1,\dotsc,j_m)$ is a subsequence of $(i_1,\dotsc,i_n)$. There is a length function on $\wt W$ defined by $\ell(x)=n$ if $s_{i_1}\dotsm s_{i_n}\tau$ is a reduced expression for $x$. It satisfies $\ell(xy)\le\ell(x)+\ell(y)$. The dimension of $IxI/I\subset \Flag$ is equal to $\ell(x)$.

\begin{lemma}\label{lemprep2}
For $x\in\wt W$ let $IxI$ denote the locally closed subscheme of $LG$ whose $k$-valued points are $I(k)xI(k)$. Let $S$ be a scheme. Let $g\in(IxI)(S)$. Then \'etale locally on $S$, there are elements $i_1,i_2$ of $I(S)$ with $g=i_1xi_2$. That is, $IxI$ is the \'etale sheaf associated with the presheaf $S\mapsto I(S)xI(S)$.
\end{lemma}
\begin{proof}
We consider the morphism of ind-schemes $I/(I\cap xIx^{-1})\rightarrow \Flag$ given by $g\mapsto gx$. The image consists of the (open) Schubert cell $S_x=IxI/I$. The morphism $I/(I\cap xIx^{-1})\rightarrow S_x$ is a bijection on closed points and on tangent spaces, and both schemes are homogenous spaces and hence smooth. Thus $I/(I\cap xIx^{-1})\rightarrow \Flag$ is an immersion with image $S_x$. Let $g\in IxI(S)$ and let $\ol g$ be its image in $\Flag$. Then by what we just saw, $\ol g$ is the image of some $i\in (I/(I\cap xIx^{-1}))(S)$. By Corollary \ref{corseckn}(b) we can lift $i$ to an element of $I$. This element satisfies $x^{-1}i^{-1}g\in I(S)$, the lemma follows.
\end{proof}

\section{Local $G$-shtukas}

In this section we review and slightly generalize the notions of (bounded) local $G$-shtukas with level structure and their Newton points from \cite{HV1}.
Let $S$ be a connected $\BF_q$-scheme and let $\mu$ be a dominant coweight of $G$. 

\begin{definition}\label{DefLocalGShtuka}
Let $K\subset K_0$ be a subgroup scheme containing some $K_n$. Then a \emph{local $G$-shtuka with $K$-structure} over an $\BF_q$-scheme $S$ is a pair $\ul\CG=(\CG,\phi)$ where $\CG$ is a $K$-torsor over $S$ for the \'etale topology and $\phi:\s \CL\CG\isoto \CL\CG$ is an isomorphism of $LG$-torsors over $S$, where $\CL\CG$ is the $LG$-torsor associated with $\CG$.
\end{definition}

Later on we are mainly interested in the two cases $K=K_0$ and $K=I$. For $P$ a parabolic subgroup of $G$ we also consider local $P$-shtukas with $K\cap P$-structure. They are defined analogously as $K\cap P$-torsors for the \'etale topology together with a Frobenius map $\phi$ on the associated $LP$-torsor.

\begin{definition}\label{DefQIsog} 
A \emph{quasi-isogeny} between local $G$-shtukas with $K$-structure $(\CG,\phi)\to(\CG',\phi')$ over $S$ is an isomorphism of the associated $LG$-torsors $f:\CL\CG\isoto \CL\CG'$ with $\phi'\s(f)=f\phi$. The set of quasi-isogenies between $\ul\CG$ and $\ul\CG'$ over $S$ is denoted $\QIsog_S(\ul\CG,\ul\CG')$.
\end{definition}

Let $k$ be an algebraically closed field and $b\in LG(k)$. Then the quasi-isogeny group $J_b\bigl(\BF_q\dpl z\dpr\bigr)$ of the local $G$-shtuka with $K$-structure $\bigl(K_k,b\s\bigr)$ can be described analogously to the situation for $p$-divisible groups as follows.

\begin{proposition}[{Compare~\cite[Propositions 1.12 and 1.16]{RZ}}]\label{PropFieldOfDefQIS}
Let $b_1,b_2\in LG(k)$. Then the functor on the category of $\BF_q\dpl z\dpr$-algebras
\[
Q(R)\es:=\es\bigl\{\,g\in G\bigl(R\otimes_{\BF_q\dpl z\dpr}k\dpl z\dpr\bigr):\es gb_1\;=\;b_2\s(g)\,\bigr\}
\]
is representable by a smooth affine scheme over $\BF_q\dpl z\dpr$. For $b=b_1=b_2$ it is even a group scheme $J_b:=Q$.

Assume that $b_1,b_2\in LG(k')$ for an algebraically closed subfield $k'\subset k$ and let $Q'$ be the corresponding functor for $k'$. Then the canonical morphism $Q'\to Q$ is an isomorphism.\qed
\end{proposition}

In particular if $\ul\CG_i=(K_k,b_i\s)$ then $\QIsog(\ul\CG_1,\ul\CG_2)=Q\bigl(\BF_q\dpl z\dpr\bigr)$. Likewise the quasi-isogeny group $\QIsog\bigl(K_k,b\s\bigr)$ equals $J_b\bigl(\BF_q\dpl z\dpr\bigr)$. If $b_1,b_2,b\in LG(k')$ then all these quasi-isogenies are defined over $k'$.

\begin{corollary}[{Compare~\cite[Corollary 1.14]{RZ}}]\label{PropQISGpDecent}
Assume that $b\in LG(k)$ satisfies a decency equation of the form
\[
(b\cdot\s)^s\es:=\es b\cdot\s(b)\cdot\ldots\cdot(\sigma^{s-1})^\ast(b)\cdot(\sigma^s)^\ast\es=\es z^{\gamma}\cdot(\sigma^s)^\ast
\]
where $s>0$ is an integer and $\gamma\in X_\ast(\T)$, hence satisfies $\frac{1}{s}\gamma_{\dom}=\nu$ in $X_\ast(\T)_\BQ$ where $\nu$ is the Newton point of $b$. Then $J_b$ is an inner form of the centralizer of $\nu$ in $G_{\BF_{q^s}\dpl z\dpr}$, a Levi subgroup of $G_{\BF_{q^s}\dpl z\dpr}$. In particular $\QIsog\bigl(K_k,b\s\bigr)\subset G\bigl(\BF_{q^s}\dpl z\dpr\bigr)=LG(\BF_{q^s})$.
\qed
\end{corollary}

\noindent
{\it Remark.}
In the sequel we will shorten (and abuse) notation and simply write $J_b$ for the quasi-isogeny group $J_b\bigl(\BF_q\dpl z\dpr\bigr)$ of $(K_k,b\s)$.

\medskip

Let $\ol{B}\subset G$ be the Borel subgroup opposite to $B$. For a dominant weight $\lambda$ of $G$ let $V_\lambda:=\bigl(\Ind_{\ol{B}}^G(-\lambda)_\dom\bigr)\dual$ be the Weyl module of $G$ with highest weight $\lambda$. It is a cyclic $G$-module generated by a $B$-stable line on which $B$ acts through $\lambda$. Any other such $G$-module is a quotient of $V_\lambda$, see for example \cite[II.2.13]{Jantzen}. For a $K$-torsor $\CG$ on a scheme $S$ we denote by $\CG_\lambda$ the \fpqc-sheaf of $\CO_S\dbl z\dbr$-modules on $S$ associated with the presheaf
\[
\Test\;\longmapsto\;\Bigl(\CG_{K_0}(\Test)\times\bigl(V(\lambda)\otimes_{\BF_q}\CO_S\dbl z\dbr(\Test)\bigr)\Bigr)\big/K_0(\Test)\,;
\]
compare \cite[Section 3]{HV1}. Here $\CG_{K_0}$ is the $K_0$-torsor associated with $\CG$. 

\begin{definition}\label{DefBounded}
\begin{enumerate}
\item 
Let $\CG$ and $\CH$ be $K$-torsors on $S$ and let $\delta:\CL\CH\isoto\CL\CG$ be an isomorphism of the associated $LG$-torsors. The isomorphism $\delta$ is \emph{bounded by $\mu$} if for each dominant weight $\lambda$ of $G$
\begin{eqnarray*} 
\delta(\CH_\lambda)&\subset&z\;^{-\langle\,(-\lambda)_\dom,\mu\rangle}\,\CG_\lambda\es\subset\es\CG_\lambda\otimes_{\CO_S\dbl z\dbr}\CO_S\dpl z\dpr\quad\text{and}\label{EqBounded1}\\
~\overline{\mu}&=&\overline{\mu}_\delta(s) \text{ in }\pi_1(G)\text{ for all }s\in S.\label{EqBounded2}
\end{eqnarray*}
Here $\mu_\delta(s)$ is the unique dominant element with $g\in K_0(k(s))z^{\mu_\delta(s)} K_0(k(s))$ for some $g\in LG(k(s))$ with $\delta(\CH)=g(\CG)\subset \CL\CG$.
\item 
A local $G$-shtuka with $K$-structure $(\CG,\phi)$ over $S$ is \emph{bounded by $\mu$} if the isomorphism 
\[
\phi:\s\CL\CG\isoto\CL\CG
\]
is bounded by $\mu$.
\item Let $x\in\wt{W}$. A local $G$-shtuka with $I$-structure $\ul\CG$ over $S$ is \emph{of affine Weyl type $x$} if \'etale-locally on $S$, it is isomorphic to a trivial $I$-torsor $I_S$ with $\phi=b\s$ for some $b\in I(S)xI(S)$. 
\end{enumerate}
\end{definition}

Let $\ul\CG$ be a local $G$-shtuka with $K$-structure over an algebraically closed field $k$. Then $\CG_{K_0}$ is isomorphic to the trivial $K_0$-torsor on $k$ and the Frobenius is given by some element $b\in LG(k)$. By \cite[Lemma 3.12]{HV1} there exists some coweight $\mu$ such that $\ul\CG$ is bounded by $\mu$. The least such coweight $\mu_{\ul\CG}$ is the Hodge polygon or Hodge point of $\ul\CG$. It is the unique dominant coweight with $b\in K_0(k)z^{\mu}K_0(k)$. In \cite[Theorem 5.6, Proposition 5.9, Theorem 9.5, Remark 9.4 (b), and Proposition 3.16]{HV1} (for the case where $\zeta=0$) we proved the following proposition for $K=K_0$, respectively $K=I$. Its generalization to arbitrary $K$ is proved in the same way.

\begin{proposition}\label{PropDefo}
Let $\ul\BG=(K_{k'},b\s)$ be a local $G$-shtuka with $K$-structure over a field $k'$ bounded by $\mu$ (respectively of affine Weyl type $y$ if $K=I$). Then the universal deformation of $\ul\BG$ by local $G$-shtukas with $K$-structure bounded by $\mu$ (respectively of affine Weyl type $y$) is pro-representable by a complete noetherian local ring $\ol\CalD$ of dimension $\langle2\rho,\mu\rangle$ (respectively of dimension $\ell(y)$, the length of $y$). The universal local $G$-shtuka with $K$-structure over $\Spf\ol\CalD$ comes from a local $G$-shtuka with $K$-structure over $\Spec\ol\CalD$.\qed
\end{proposition}

\bigskip

Finally we will need the following fact which is a straightforward generalization of \cite[Theorem 6.2]{HV1} (in the case where $\zeta=0$).

\begin{proposition}\label{PropHV5.2}
Let $X$ be an $\BF_q$-scheme. Let $b\in LG(X)$ and consider the local $G$-shtuka with $K$-structure $\CH=(K_X,b\s)$ over $X$. Then the ind-scheme $(LG/K)\times_{\BF_q} X$ pro-represents the functor from $X$-schemes to sets
\begin{eqnarray*}
S&\longmapsto & \Bigl\{\,\text{Isomorphism classes of pairs }(\ul\CG,\rho)\text{ where }\ul\CG \text{ is a local $G$-shtuka} \\
&&\qquad  \text{with $K$-structure over $S$ and }\rho:\ul\CG\to\ul\CH_{S}\text{ is a quasi-isogeny}\,\Bigr\}
\end{eqnarray*}
Here $(\ul\CG,\rho)$ and $(\ul\CG',\rho')$ are called isomorphic if there is an isomorphism $\alpha:\ul\CG\isoto\ul\CG'$ with $\rho=\rho'\circ\alpha$. \qed
\end{proposition}

%%%%%%%%%%%%%%%%%%%%%%%%%%%%%%%%%%%%%%%%%%%%%%%%%%%%%%%%%%%%%%%%%%%%%%
%
%   Completely slope divisible local $G$-shtukas
%
%%%%%%%%%%%%%%%%%%%%%%%%%%%%%%%%%%%%%%%%%%%%%%%%%%%%%%%%%%%%%%%%%%%%%%

\section{Completely slope divisible local $G$-shtukas}\label{seccsd}
 
For a parabolic subgroup $P\subset G$ containing $\T$ but not necessarily $B$ let $M$ be the Levi factor containing $T$ and let $N$ be the unipotent radical of $P$. Let $\ol P$ be the opposite parabolic and $\ol N$ its unipotent radical. Denote by $\wt W_M$ the extended affine Weyl group of $M$. It is canonically a subgroup of $\wt W$.

Let $K$ be a subgroup scheme of $I$. We define the group schemes $K_M$, $K_N$, and $K_{\ol N}$ by their values on $\BF_q$-algebras $R$
\[
K_M(R):=K(R)\cap M\bigl(R \dbl z\dbr\bigr)\,,\quad K_N(R):=I(R)\cap N\bigl(R\dbl z\dbr\bigr)\,,\quad K_\ol{N}(R):=I(R)\cap\ol N\bigl(R\dbl z\dbr\bigr)\,.
\]

Let us recall the definition and some of the main properties of fundamental alcoves from \cite{GHKR2}.

\begin{definition}[{\cite[Definition 13.1.2]{GHKR2}}]\label{DefPFund}
An element $x\in \wt W$ is called \emph{$P$-fundamental}, or a \emph{fundamental $P$-alcove} if $x\in\wt W_M$ and if
$xI_{M}x^{-1}=I_{M}$, $xI_{N}x^{-1}\subset I_{N}$, and $x^{-1}I_{\ol N}x\subset I_{\ol N}$.
\end{definition}

\begin{remark}\label{remfundalc}
\begin{enumerate}
\item \cite[Corollary 13.2.4]{GHKR2}\label{a} Let $b\in LG(k)$ for some algebraically closed field $k$. Then there is an $x\in \wt W$ which is $P$-fundamental for some $P$ and such that $b=g^{-1}x\s (g)$ for some $g\in LG(k)$.
\item \cite[Proposition 6.3.1]{GHKR2} \label{e} If $x$ is $P$-fundamental, and $g\in I(k)xI(k)$ then there is an $h\in I(k)$ with $h^{-1}g\s(h)=x$.
\item\label{b}\cite[Corollary 13.2.4]{GHKR2}, \cite[Lemma 9]{trunc1} All elements $x$ as in \ref{a} are obtained as follows. Let $\nu$ be the Newton point of $b$ and let $M_\nu$ be the centralizer of $\nu$. Then the $\sigma$-conjugacy class $[b]$ of $b$ contains a uniquely determined element $x_b\in \wt W_{M_\nu}$ such that $x_b$ has length zero with respect to $M_\nu$ and $I\cap M_{\nu}$, and the $M_\nu$-dominant Newton point of $x_b$ equals $\nu$. It is called the \emph{standard representative} of $[b]$. Let $P$ be the standard parabolic subgroup with Levi component $M_{\nu}$. Then there is an element $w\in W$ of minimal length in its coset $W_{M_{\nu}}\cdot w$ such that $x:=w^{-1}x_b w\in [b]$ is $w^{-1}Pw$-fundamental, and all $x\in[b]$ that are $P$-fundamental for some $P$ arise in this way.
\item\label{c}\cite[Lemma 9 (c)]{trunc1} Let $x$ be $P$-fundamental, and $M'$ the centralizer of its $M$-dominant Newton point. Then $M'\supseteq M$ and $x$ is also $P'$-fundamental for $P'=M'P$.
\item\label{d} Let $x\in\wt W_M$ be $P$-fundamental for $P=MN$ and assume that $M$ is the centralizer of the $M$-dominant Newton point $\nu_x$ of $x$. Then for any $d\in\BN$ there is an $\el \in\BN$ with 
\[
x^{\el }I_Nx^{-\el }\;,\;x^{-\el }I_{\ol N}x^\el \;\subset\;I_d.
\]
Indeed, let $t>0$ such that $x^t\in X_\ast(\T)\subset\wt W_M$. Then by our choice of $M$, we have $\langle\alpha,x^t\rangle= \langle\alpha,t\nu\rangle>0$ for all roots $\alpha$ of $N$ and $\langle\bar\alpha,x^t\rangle<0$ for all roots $\bar\alpha$ of $\ol N$. 
\end{enumerate}
\end{remark}

\begin{lemma}\label{LemmaBoundedN}
Let $x$ be a $P$-fundamental alcove where $P$ is such that the Levi subgroup $M$ of $P$ equals the centralizer of the $M$-dominant Newton point of $x$. Let $\ol N$ be the unipotent radical of the parabolic opposite to $P$ and let $\bar n\in L\ol N(S)$ for a quasi-compact scheme $S$. Then there exists an integer $\el_0\in\BZ$ with $\bar n\in x^{-\el_0}I_{\ol N}(S)x^{\el_0}$. In particular $\bar n$ is bounded.
\end{lemma}

\begin{proof}
Since $S$ is quasi-compact $L\ol N(S)=\ol N\bigl(\Gamma(S,\CO_S\dpl z\dpr)\bigr)=\ol N\bigl(\Gamma(S,\CO_S)\dbl z\dbr[\frac{1}{z}]\bigr)$. Again by the quasi-compactness of $S$, the argument given in Remark~\ref{remfundalc}\ref{d} implies that $x^{t\el}\bar nx^{-t\el}\in I_{\ol N}$ for $\el\ll0$. Also the boundedness of $\bar n$, that is of the automorphism $\delta$ of $LG_S$ given by left multiplication by $\bar n$, follows from this and from the fact that $\bar\mu_{\delta}(s)=0$ for all $s\in S$.
\end{proof}

\begin{example}\label{ExFundAlcoveGLn}
Let $G=\GL_5$ with $B$ being the Borel of upper triangular matrices, and let $b=(12)(354)z^{(1,0,1,0,0)}$. Then $b$ has Newton point $(\frac{1}{2},\frac{1}{2},\frac{1}{3},\frac{1}{3},\frac{1}{3})$ and is already the standard representative of its class $[b]$. In this example, the only fundamental alcove in $[b]$ is $x=(13)(254)z^{1,1,0,0,0}=w^{-1}bw$ for $w=(23)$. In particular, we see that we cannot choose $x$ to be $P$-fundamental for a standard $P$.
\end{example}

\begin{lemma}\label{lemprep1}
Let $b\in LG(k)$ and let $x$ be a $P$-fundamental alcove in the $\sigma$-conjugacy class $[b]$. Then $IxI\subset [b]$ is closed in $[b]$. 
\end{lemma}
\begin{proof} Recall that the closure of $IxI$ is the union of all $IyI$ with $y\leq x$ in the Bruhat order. Assume that there is some $g\in IyI$ with $y\leq x$ and $g\in [b]$. Then by \cite[Corollary 12.1.2]{GHKR2} or \cite[Proposition 12]{trunc1}, there is a $w\in W$ with $w^{-1}xw\leq y$. By \cite[Proposition 13.1.2]{GHKR2} we have $\ell(x^n)=n\langle 2\rho,\nu_b\rangle=n\ell(x)$ for each $n\in\BN$ where $\rho$ is the half-sum of the positive roots and where $\nu_b$ is the Newton point of $b$. Furthermore $n\ell(w^{-1}xw)\geq\ell((w^{-1}xw)^n)=\ell(w^{-1}x^nw)\geq n\ell(x)-2\ell(w)$ for each $n$, hence $\ell(w^{-1}xw)\geq\ell(x)$. Thus $w^{-1}xw=y=x$.
\end{proof}

The following lemma is a scheme-theoretic version of a special case of the Iwahori decomposition.

\begin{lemma}\label{LemmaIwahoriDecomp} Let $K$ be a subgroup scheme of $I$ such that $K(S)\supset\T\bigl(\Gamma(S,\CO_S)\dbl z\dbr\bigr)$ for every $\BF_q$-scheme $S$. Then every element $g\in K(S)$ has uniquely determined decompositions $g=nm\bar n$ and $g=\bar n'm'n'$ for elements $m,m'\in K_M(S)$, $n,n'\in K_N(S)$, and $\bar n,\bar n'\in K_\ol{N}(S)$.
\end{lemma}

\begin{proof}
As $K(k)\supset\T(k\dbl z\dbr)$, the group scheme $K$ is generated by $\T(k\dbl z\dbr)$ and all affine root subgroups in $K$. There is a filtration $K=G_1\supset G_2\supset G_3\supset\dotsm$ such that $G_2$ is the unipotent radical of $K$, each $G_i$ is a normal subgroup of $K$ defined over $\BF_q$ and for $i\geq 2$, we have that $G_i/G_{i+1}\cong \mathbb{G}_a$ is generated by one of the affine root subgroups. Indeed, it can be obtained by intersecting a suitable filtration of $I$ by Moy-Prasad subgroups with $K$. We first decompose $g$ as $g_1g_2$ with $g_1\in \T\bigl(\Gamma(S,\CO_S)\dbl z\dbr\bigr)\subset K_M(S)$ and $g_2\in G_2\bigl(\Gamma(S,\CO_S)\dbl z\dbr\bigr)$. As each affine root subgroup is contained in one of the subgroups $K_M, K_N, K_{\ol N}$, one can then inductively construct decompositions of $g$ modulo some $G_i$. By passing to the limit one obtains the decompositions above. The uniqueness follows from the fact $\ol{N}\cap MN=\{1\}$.
\end{proof}

\begin{definition}\label{DefSlopeFilt1}
A bounded local $G$-shtuka with $I$-structure $\ul\CG$ over an $\BF_q$-scheme $S$ is called \emph{completely slope divisible} if there exists a parabolic subgroup $P$ of $G$ containing $\T$, a local $\ol P$-shtuka with $I\cap L\ol P$-structure $\ul{\bar\CP}=(\bar\CP,\phi)$ over $S$, and a $P$-fundamental $x$ such that
\begin{itemize}
\item $\ul\CG\cong i_\ast\ul{\bar\CP}$ where $i:\ol P\hookrightarrow G$ is the inclusion morphism,
\item \'etale-locally on $S$ the Frobenius $\phi$ is represented by an element of $x\cdot(L\ol P\cap I)(S)$.
\end{itemize}
\end{definition}
Note that by Remark~\ref{remfundalc} \ref{c} one can always enlarge $M,P$ and $\ol P$ and assume that $M$ equals the centralizer of the $M$-dominant Newton point of $x$.
See Section \ref{secgln} for a comparison to the corresponding (but slightly weaker) notion for $p$-divisible groups defined in \cite[Def 1.2]{OortZink}. 

\begin{proposition}\label{PropSlopeFilt}
Let $x$ be $P$-fundamental for a parabolic $P=MN\supset\T$ in which $M$ is the centralizer of the $M$-dominant Newton point of $x$. Let $g\in I(S)xI(S)$ for an $\BF_q$-scheme $S$. Then there is an element $h\in I(S)$ with $h^{-1}g\s(h)=xp$ for $p\in I_M(S)I_\ol{N}(S)$. In other words, the local $G$-shtuka with $I$-structure defined by $g$ is completely slope divisible.
\end{proposition}

\begin{proof}
By induction on $\el $ we prove the following

\smallskip

\noindent
{\it Claim.} For any $\el \in\BN$ there exists an element $h_\el \in I(S)$ with $h_\el ^{-1}g\s(h_\el )=xp_\el  r_\el $ for elements $p_{\el}\in I_M(S)I_\ol{N}(S)$ and $r_{\el}\in x^{\el} I_N(S)x^{-{\el}}$. We can choose the $h_{\el}$ such that $h_{{\el}-1}^{-1}h_{\el}\in x^{\el} I_N(S)x^{-{\el}}$.

\smallskip

To prove the claim for ${\el}=0$ we write $g=h_0xh'_0$ with $h_0,h'_0\in I(S)$. Then $h_0^{-1}g\s(h_0)=xh'_0\s(h_0)$ with $h'_0\s(h_0)\in I(S)=I_M(S)I_\ol{N}(S)I_N(S)$ by Lemma~\ref{LemmaIwahoriDecomp}. Now assume that the claim holds for ${\el}$. By Lemma~\ref{LemmaIwahoriDecomp} for the group $K=I\cap x^{\el} Ix^{-{\el}}=(x^{\el} I_N x^{-{\el}}) I_MI_{\ol N}$ we write $p_{\el} r_{\el}=n_{\el} m_{\el} \bar n_{\el}$ with $n_{\el}\in x^{\el} I_N(S)x^{-{\el}}$, $m_{\el}\in I_M(S)$, and $\bar n_{\el}\in I_\ol{N}(S)$. Setting $h_{{\el}+1}:=h_{\el} xn_{\el} x^{-1}$ we compute
\[
h_{{\el}+1}^{-1}g\s(h_{{\el}+1})\;=\;xn_{\el}^{-1}x^{-1}h_{\el}^{-1}g\s(h_{\el})x\s(n_{\el})x^{-1}\;=\;xm_{\el}\bar n_{\el} x\s(n_{\el})x^{-1}\;=\;xp_{{\el}+1}r_{{\el}+1}
\]
for $p_{{\el}+1}=m_{\el}\bar n_{\el}\in I_M(S)I_\ol{N}(S)$ and $r_{{\el}+1}=x\s(n_{\el})x^{-1}\in x^{{\el}+1}I_N(S)x^{-({\el}+1)}$. This proves the claim.

By Remark \ref{remfundalc} \ref{d}, we have that for every $d\in\BN$ there is an ${\el}\in\BN$ such that for all integers ${\el}'\ge {\el}$
\[
h_{\el}^{-1}h_{{\el}'}\;\in\; x^{{\el}}I_N(S)x^{-{\el}}\subset K_d.
\]
This implies that the sequence $h_{\el}$ converges to an element $h\in I(S)$ in the $z$-adic topology. Also $M\ol N\cap N=\{1\}$ implies that $r_{{\el}'}\equiv r_{\el}\equiv 1\pmod{z^d}$ and $p_{{\el}'}\equiv p_{{\el}}$. Hence the $p_{\el}$ converge to an element $p\in I_M(S)I_\ol{N}(S)$ and the $r_{\el}$ converge to $1$. Altogether we find $h^{-1}g\s(h)=xp$ as desired.
\end{proof}

The proposition implies the following characterization of complete slope divisibility (use Remark~\ref{remfundalc} \ref{d}).

\begin{corollary}\label{CorCSD}
A bounded local $G$-shtuka with $I$-structure $\ul\CG$ over an $\BF_q$-scheme $S$ is completely slope divisible if and only if there is a $P$-fundamental element $x\in\wt W$ for some parabolic $P\supset\T$ such that \'etale-locally on $S$, $\ul\CG$ is isomorphic to $(I_S,g\s)$ with $g\in I(S)xI(S)$.\qed
\end{corollary}

\begin{remark}
Let $\ul\CG$ be a completely slope divisible local $G$-shtuka with $I$-structure over an algebraically closed field $k$ and let $x$ be the associated fundamental alcove. Then by Remark \ref{remfundalc} \ref{e}, there is a trivialization of the $I$-torsor such that the Frobenius is given by $x\s$.
\end{remark}

\begin{lemma}\label{LemmaTrivN}
Let $x\in\wt W_M$ be $P$-fundamental for $P=MN$ where $M$ is the centralizer of the $M$-dominant Newton point of $x$. Let $\bar n\in L\ol N(S)$ for a noetherian $\BF_q$-scheme $S$ and fix some $d\in\BN$. Then $\bar n$ is bounded, i.e.~contained in $x^{-{\el_0}}I_{\ol N}(S)x^{\el_0}$ for some $\el_0\in\BZ$, and there is a finite surjective radicial morphism $S'\to S$ and an element $h\in x^{-{\el_0}}I_{\ol N}(S')x^{\el_0}$ with $h^{-1}x\bar n\s(h)=x\bar n'$ for $\bar n'\in I_{\ol N}(S')\cap I_d(S')$. If $S$ is local with closed point $s$ and if the reduction of $\bar n$ in $s$ is equal to $1$, then $h$ can be chosen such that its reduction in $s$ is also equal to $1$.
\end{lemma}

\begin{proof}
The boundedness of $\bar n$ and the existence of $\el_0\in\BZ$ with $\bar n\in x^{-{\el_0}}I_{\ol N}(S)x^{\el_0}$ follows from Lemma~\ref{LemmaBoundedN}. We use induction on $l\geq l_0$ and construct finite coverings $S_{\el}\to S$ and elements $h_{\el}\in x^{-{\el_0}}I_{\ol N}(S_\el)x^{\el_0}$ with $h_{\el}^{-1}x\bar n\s(h_{\el})=x\bar n_{\el}$ such that $\bar n_{\el}=x^{-{\el}}y_{\el} x^{\el}\in x^{-{\el}}I_{\ol N}(S_{\el})x^{\el}$. Choosing ${\el}$ big enough suffices for our purpose by Remark \ref{remfundalc} \ref{d}. We start with $S_{l_0}=S$, $h_{l_0}=1$, and $\bar n_{l_0}=\bar n$. Assume we already found $S_{\el}$ and $h_{\el}$. Since $I_{\ol N}/x^{-1}I_{\ol N}x$ is a scheme of finite type over $\BF_q$, its $q$-Frobenius endomorphism $\sigma$ is finite surjective. Consider the cartesian diagram in which the lower horizontal morphism is given by the element $y_{\el}^{-1}\in I_{\ol N}(S_{\el})$
\[
\xymatrix @C+3pc @R+1pc {
S_{{\el}+1} \ar[r]^{\TS \tilde y_{{\el}}^{-1}\quad} \ar[d] & I_{\ol N}/x^{-1}I_{\ol N}x \ar[d]^{\TS \sigma}\\
S_{\el} \ar[r]^{\TS y_{{\el}}^{-1}\quad} & I_{\ol N}/x^{-1}I_{\ol N}x
}
\]
Then $S_{{\el}+1}\to S_{\el}$ is also a finite surjective radicial morphism. The upper horizontal morphism is given by an element $\tilde y_{\el}^{-1}\in I_{\ol N}(S_{{\el}+1})$ with $\s(\tilde y_{\el}^{-1})=y_{\el}^{-1}(x^{-1}f_{\el} x)$ for an $f_{\el}\in I_{\ol N}(S_{{\el}})$. Set $h_{{\el}+1}=h_{\el} x^{-{\el}}\tilde y_{\el}^{-1}x^{{\el}}$. Then
\begin{eqnarray*}
h_{{\el}+1}^{-1}x\bar n\s(h_{{\el}+1})&=&x^{-{\el}}\tilde y_{\el} x^{\el} xx^{-{\el}}y_{\el} x^{\el} x^{-{\el}}\s(\tilde y_{\el}^{-1})x^{\el}\\[2mm]
&=& x(x^{-({\el}+1)}\tilde y_{\el}f_{\el} x^{{\el}+1}).
\end{eqnarray*}
Let $y_{{\el}+1}:=\tilde y_{\el}f_{\el}\in I_{\ol N}(S_{{\el}+1})$. Then $$\bar n_{{\el}+1}=x^{-({\el}+1)}y_{{\el}+1}x^{{\el}+1}\in x^{-({\el}+1)}I_{\ol N}(S_{{\el}+1})x^{{\el}+1}$$ as desired. Let $\el$ be big enough such that $x^{-l}I_{\ol N}x^l\subset I_d$. We stop the iteration at $l$ to find a finite surjective radicial morphism $S':=S_\el\to S$, an element $h:=h_l\in x^{-{\el_0}}I_{\ol N}(S')x^{\el_0}$ with $h^{-1}x\bar n\s(h)=x\bar n'$ for $\bar n':=\bar n_l\in I_{\ol N}(S')\cap I_d(S')$. The last assertion also follows from the explicit construction of the elements $h_l$.
\end{proof}

Over a perfect field $k$ the proof has the following

\begin{corollary}\label{Cor4.10'}
Let $x\in\wt W_M$ be $P$-fundamental for $P=MN$ where $M$ is the centralizer of the $M$-dominant Newton point of $x$. Let $\bar n\in x^{-{\el_0}}I_{\ol N}(k)x^{\el_0}$ for a perfect field $k$ and some $\el_0\in\BZ$. Then there is an element $h\in x^{-{\el_0}}I_{\ol N}(k)x^{\el_0}$ with $h^{-1}x\bar n\s(h)=x$. 
\end{corollary}

\begin{proof}
In the proof of Lemma~\ref{LemmaTrivN} we have $S_\el=\Spec k$ for all $\el$ with the morphisms $S_{\el+1}\to S_\el$ being the Frobenius automorphism. Here we do not stop the iteration. Instead we let $h$ be the limit of the $h_l$. It exists because the $x^{-l}\tilde y_l^{-1}x^l\in x^{-l}I_{\ol N}(k)x^l$ and the $\ol n_\el\in x^{-l}I_{\ol N}(k)x^l$ converge to 1. 
\end{proof}

\begin{proposition}\label{PropOZ2.5}
Let $\ul\CH$ be a local $G$-shtuka with $K$-structure over an algebraically closed field $k$. Let $\mu\in X_\ast(\T)$ be dominant. Then up to isomorphism there exist only finitely many quasi-isogenies $\rho:\ul\CG\to\ul\CH$ bounded by $\mu$ with $\ul\CG$ a completely slope divisible local $G$-shtuka with $K$-structure. Here we say that $\rho:\ul\CG\to\ul\CH$ and $\rho':\ul\CG'\to\ul\CH$ are isomorphic if there is an isomorphism $\alpha:\ul\CG\isoto\ul\CG'$ with $\rho=\rho'\circ\alpha$.
\end{proposition}

\begin{proof}
By Remark \ref{remfundalc} \ref{a} there is a quasi-isogeny $\eta:\ul{\CH}\to\ul{\wt\CH}$ with $\ul{\wt\CH}=(K_k,x\s)$ for some $P$-fundamental $x$. By \cite[Lemma 3.11]{HV1} the quasi-isogeny $\eta$ is bounded by its Hodge point $\mu_\eta$. If $\rho:\ul\CG\to\ul\CH$ is a quasi-isogeny bounded by $\mu$ then $\eta\circ \rho$ is bounded by $\mu+\mu_\eta$. Hence we may assume that $\ul\CH=\ul{\wt \CH}$. Note that there are only finitely many possibilities for the parabolic subgroup $P$. By Corollary~\ref{CorCSD} and Remark \ref{remfundalc} \ref{e} and \ref{c} we may assume that $\ul\CG=(K_k,x'\s)$ for some $P'$-fundamental $x'$ such that $P'$ is the centralizer of the $M'$-dominant Newton point of $x'$. As $\ul\CG$ and $\ul\CH$ are isogenous, Remark \ref{remfundalc} \ref{b} and \ref{c} imply that we may assume that $x=x'$ and $P=P'$. Then the quasi-isogeny is given by an element $h\in \QIsog(K_k,x\s)$ which is bounded by $\mu$. Since $x$ is decent, Corollary~\ref{PropQISGpDecent} shows that there are only finitely many such quasi-isogenies.
\end{proof}

\begin{theorem}\label{ThmOZ2.4}
Let $S$ be an integral noetherian $\BF_q$-scheme and let $g\in LG(S)$ such that the $\sigma$-conjugacy classes of $g_s$ in the geometric points $s$ of  $S$ all coincide. Let $x$ be a $P$-fundamental alcove associated with this $\sigma$-conjugacy class as in Remark \ref{remfundalc} \ref{a}. Then there are morphisms $\wt S\xrightarrow{\beta} S'\xrightarrow{\alpha} S$  with $S'$ integral, $\alpha$ finite surjective, and $\beta$ an \'etale covering, and a bounded element $\tilde h\in LG(\wt S)$ satisfying $\tilde h^{-1}g_{\wt S}\s(\tilde h)\in I(\wt S)xI(\wt S)$.
\end{theorem}
The Theorem is an analog of \cite[Lemma 2.4]{OortZink}.

\begin{proof}
Let $\ol L$ be an algebraic closure of the function field $L$ of $S$. By Remark \ref{remfundalc} \ref{a} there is an $\bar h\in LG(\ol L)$ with $\bar h^{-1}g_{\ol L}\,\s(\bar h)=x$. The image of $\bar h$ in the affine flag variety lies inside a closed Schubert cell $\ol C\subset\Flag$. Note that $\ol C$ is a scheme which is projective over $\BF_q$. Hence by multiplying $\bar h$ with an element of $I(\ol L)$ on the right we may assume that there is a finite field extension $L'/L$ and an $h'_{L'}\in LG(L')$ with $(h'_{L'})^{-1}g_{L'}\,\s(h'_{L'})\in I(L')xI(L')$. Let $ S'$ be the scheme theoretic closure of the $L'$-valued point $h'_{L'}$ inside $\ol C\times_{\BF_q}S$, and let $\alpha:S'\to S$ and $h':S'\to\ol C\subset\Flag$ be the projections. Then $S'$ is irreducible and reduced. Over an \'etale covering $\beta:\wt S\to S'$ the map $h'$ is represented by an element $\tilde h\in LG(\wt S)$ due to Lemma~\ref{LemmaTildeKEtale}.

Consider the $\wt S$-valued point $\tilde h^{-1}g_{\wt S}\,\s(\tilde h)$ of the affine flag variety $\Flag$. By construction it lies generically in $IxI/I$. So it lies inside the closed Schubert cell $\ol{IxI/I}\subset \Flag$. By assumption it lies in the Newton stratum of $x$ inside $\ol{IxI/I}$. By Lemma \ref{lemprep1} $IxI/I$ is closed in this Newton stratum. Hence the point lies in $IxI$. By Lemma \ref{lemprep2} we find that $\tilde h^{-1}g_{\wt S}\,\s(\tilde h)$ lies in $I(\wt S)xI(\wt S)$.

It remains to show that the morphism $\alpha:S'\to S$ is finite. For this consider the Stein factorization $S'\to S''\to S$ with $S''=\Spec\alpha_\ast\CO_{S'}$. Then $S''\to S$ is finite and $\gamma:S'\to S''$ has connected fibers \cite[III, Corollaire 4.3.2]{EGA}. We reduce to the following

\medskip

\noindent
{\it Claim.} For any point $s\in S''$ the fiber $S'_s=\gamma^{-1}(s)$ is mapped under $h':S'\to\ol C$ to a single point in $\ol C$.

\medskip

\noindent
Indeed, by \cite[Lemma 2.6]{OortZink} the claim implies that the morphism $h':S'\to\ol C$ factors through $S''$, and thus $S'=S''$ by definition of $S'$. Let us prove the claim. Consider the local $G$-shtuka with $I$-structure $\ul\CH=(I_S,g\s)$ over $S$. By Proposition~\ref{PropHV5.2} the map $h':S'\to\Flag$ uniquely determines a local $G$-shtuka with $I$-structure $\ul\CG$ over $S'$ together with a quasi-isogeny $\rho:\ul\CG\to\ul\CH_{S'}$ up to isomorphism. Its pullback to $\wt S$ is trivialized $\CG_{\wt S}\cong(I_{\wt S},\tilde h^{-1}g_{\wt S}\,\s(\tilde h)\s\bigr)$ and $\rho$ is given by multiplication with $\tilde h$; compare the proof of \cite[Theorem 6.2]{HV1}. The quasi-isogeny $\rho$ is bounded by some dominant $\mu\in X_\ast(\T)$ because $h'$ lies in the Schubert cell $\ol C(S')$. Since $\tilde h^{-1}g_{\wt S}\,\s(\tilde h)\in I(\wt S)xI(\wt S)$, Corollary~\ref{CorCSD} tells us that $\ul\CG$ is completely slope divisible. Now fix a point $s\in S''$ and consider the base changes of $\ul\CH$ to the residue field $k(s)$, and of $\ul\CG$ to the fiber $S'_s$ over $s$. Then we are in the situation of Proposition~\ref{PropOZ2.5} and we conclude that the image of $h':S'_s\to\ol C$ is a finite set of points. Since $S'_s$ is connected, this image must be a single point in $\ol C$. This proves the claim and the theorem.
\end{proof}

\begin{corollary}\label{CorSlopeFilt}
Let $R$ be an integral noetherian complete local ring with algebraically closed residue field $k$. Let $\ul\CG$ be a  local $G$-shtuka over $\Spec R$ with $K$-structure whose quasi-isogeny class in all geometric points of $R$ is constant. Let $x$ be a $P$-fundamental alcove corresponding to this quasi-isogeny class. Let $\rho_0$ be a quasi-isogeny between the fiber of $\ul\CG$ over $k$ and $(K_k,x\s)$. Then for each fixed $d\in \BN$ there exists a finite integral extension $R''\supset R$ and a bounded quasi-isogeny $\rho:\ul\CH\to\ul\CG_{R''}$ of local $G$-shtukas with $K$-structure over $\Spec R''$ with reduction $\rho_k=\rho_0$, where $\ul\CH$ has a trivialization $\ul\CH\cong (K_{R''}, x\ol n\s)$ with $\ol n\in I_{d,\ol N}(R'')$ and reduction $\ol n_k=1$. 
\end{corollary}
Note that for $K=I$, the local $G$-shtuka with $I$-structure $\ul\CH$ is completely slope divisible, compare Corollary~\ref{CorCSD}.
\begin{proof}
Since $R$ has no non-trivial \'etale coverings, there is a trivialization $\ul\CG=(K_R,g\s)$. Thus the assertion is equivalent to the existence of $R''$ and an element $h\in LG(R'')$ with reduction $\rho_0$ over the special point such that $h^{-1}g\s(h)=x\ol n$ with $\ol n\in I_{d,\ol N}(R'')$. Note that by \cite[Lemma 3.13]{HV1} the boundedness of $h$ and $\rho$ is automatic because $R''$ is integral.

Theorem~\ref{ThmOZ2.4} yields a finite surjective morphism $S'=\Spec R'\to S=\Spec R$ with $R'$ integral and an element $h_1\in LG(R')=G\bigl(R'\dpl z\dpr\bigr)$ with $b':=(h_1)^{-1}g_{R'}\s(h_1)\in I(R')xI(R')$. Indeed, observe that the \'etale morphism $\beta:\wt S\to S'$ from Theorem~\ref{ThmOZ2.4} is an isomorphism because also $R'$ is complete with algebraically closed residue field. By Proposition~\ref{PropSlopeFilt} we may further assume that $b'=xm\bar n_1\in x\,I_M(R')I_{\ol N}(R')$ where $M$ is the centralizer of the $M$-dominant Newton point of $x$ and $P\subset MN$. Since $xm$ is basic in $M$, \cite[Proposition 8.1]{HV1} yields an element $h_2\in M(R'\dpl z\dpr)$ which $\sigma$-conjugates $xm$ to $x$. Hence $h_2^{-1}b'\s(h_2)=x\bar n'$ for $\bar n':=\s(h_2)^{-1}\bar n_1\s(h_2)\in\ol N(R'\dpl z\dpr)$. Let $ \bar n'_k$ be the reduction of $\bar n'$ in the special point. Then $\bar n'_k\in x^{-l_0}I_{\ol N}(k)x^{l_0}$ for some $l_0\in\BZ$ by Lemma~\ref{LemmaBoundedN}. Corollary~\ref{Cor4.10'} shows that there is an $h_{3,k}\in x^{-l_0}I_{\ol N}(k)x^{l_0}$ with $h_{3,k}^{-1}x\bar n'_k \s(h_{3,k})=x$. We lift the element $h_{3,k}$ to an element $h_3\in x^{-l_0}I_{\ol N}(R')x^{l_0}$. This is possible since $I_{\ol N}$ is pro-unipotent. Replacing $h_2$ by $h_2h_3$ we see that we may in addition assume that $\bar n'_k=1$. The given quasi-isogeny $\rho_0$ corresponds to an element $h\in LG(k)$ with $j_0=h^{-1}(h_1h_2)_k\in J_{x}(k)$. There is a lift $j\in J_x\cap LG(R')$ of $j_0$ because $x\in LG(\BF_q)$ is decent, hence $j_0\in LG(\BF_{q^s})$ by Corollary~\ref{PropQISGpDecent}, and $\BF_{q^s}\subset R'$ since $k$ contains $\BF_q^\alg$ and $R'$ is henselian. Replacing $h_1h_2$ by $h_1h_2j^{-1}$ we see that we may in addition assume that $(h_1h_2)_k$ yields the given isogeny $\rho_0$ between $\ul\CG_k$ and $(K_k,x\s)$. We now apply Lemma \ref{LemmaTrivN} for $\bar n'$ (with some possibly different $l_0$) and the constant $d$ and obtain a finite extension $R''$ of $R'$ and an element $h_4\in\ol N(R''\dpl z\dpr)$ with $(h_4)_k=1$ such that $h_4^{-1}x\bar n'\s(h_4)=x\ol n$ with $\ol n\in I_{d,\ol N}(R'')$. Setting $h=h_1h_2h_4$ proves the assertion.
\end{proof}

%%%%%%%%%%%%%%%%%%%%%%%%%%%%%%%%%%%%%%%%%%%%%%%%%%%%%%%%%%%%%%%%%%%%%%
%
%    The general linear group
%
%%%%%%%%%%%%%%%%%%%%%%%%%%%%%%%%%%%%%%%%%%%%%%%%%%%%%%%%%%%%%%%%%%%%%%

\section{The general linear group}\label{secgln}

In this section we compare our notion of completely slope divisible local $G$-shtukas to Zink's notion of completely slope divisible $p$-divisible groups. To make the parallel more visible we also explain the relation to local shtukas (which are described via locally free sheaves and correspond to local $G$-shtukas for $G=\GL_\rr$). We denote by $\s$ the endomorphism of $\CO_S\dbl z\dbr$ and $\CO_S\dpl z\dpr$ that acts as the identity on the variable $z$, and as $b\mapsto b^q$ on local sections $b\in \CO_S$. For a sheaf $V$ of $\CO_S\dbl z\dbr$-modules on $S$ we set $\s V:=V\otimes_{\CO_S\dbl z\dbr,\s}\CO_S\dbl z\dbr$. 

\begin{definition}
A \emph{local shtuka} over an $\BF_q$-scheme $S$ (of rank $\rr$) is a pair $\ulV=(V,\phi)$ where $V$ is a sheaf of $\CO_S\dbl z\dbr$-modules on $S$ which Zariski-locally is free of rank $\rr$, together with an isomorphism of $\CO_S\dpl z\dpr$-modules $\phi:\s V\otimes_{\CO_S\dbl z\dbr}\CO_S\dpl z\dpr\isoto V\otimes_{\CO_S\dbl z\dbr}\CO_S\dpl z\dpr$.

A \emph{quasi-isogeny} between local shtukas $(V,\phi)\to(V',\phi')$ is an isomorphism of $\CO_S\dpl z\dpr$-modules $f:V\otimes_{\CO_S\dbl z\dbr}\CO_S\dpl z\dpr\isoto V'\otimes_{\CO_S\dbl z\dbr}\CO_S\dpl z\dpr$ with $\phi'\s (f)=f\phi$.
\end{definition}

\begin{lemma}[{\cite[Lemma 4.2]{HV1}}] \label{LemmaHV3.2}
The quasi-isogeny categories of local $\GL_\rr$-shtukas with $K_0$-structure over $S$ and of local shtukas over $S$ are equivalent. In particular the trivialized local $\GL_\rr$-shtuka with $K_0$-structure $(K_{0,S},g\s)$ over $S$ corresponds to the trivialized local shtuka $(\CO_S\dbl z\dbr^{\oplus \rr},\phi=g\s)$. \qed
\end{lemma}

\begin{definition}\label{DefSlopeFilt2}
For general $G$ a local $G$-shtuka with $K_0$-structure $\ul\CG=(\CG,\phi)$ over an $\BF_q$-scheme $S$ is called \emph{completely slope divisible in the sense of Zink} if there exists a standard parabolic subgroup $P\subset G$ with opposite parabolic $\ol P$, a local $\ol P$-shtuka with $L\ol P\cap K_0$-structure $\ul{\bar\CP}=(\bar\CP,\phi)$ over $S$, an integer $s>0$, and a central cocharacter $\mu:\BG_m\to M$ which is $G$-dominant ($M$ is the Levi component of $P$ containing $\T$) such that
\begin{enumerate}
\item $\ul\CG\cong i_\ast\ul{\bar\CP}$ where $i:\ol P\hookrightarrow G$ is the inclusion morphism,
\item If $\ul{\bar\CP}_{S'}$ is trivializable for some $S'\rightarrow S$ then with respect to every trivialization, $z^{-\mu}\phi^s$ is represented by an element of $\ol P\bigl(\Gamma(S',\CO_{S'})\dbl z\dbr\bigr)$. 
\end{enumerate}
\end{definition}
Assume that the second condition holds for some trivialization of $\ul\CG_{S'}$. The element $b$ obtained from any other trivialization differs from the old by $\sigma$-conjugation with an element of $\ol P\bigl(\Gamma(S',\CO_{S'})\dbl z\dbr\bigr)$. Thus it automatically also satisfies the second condition. Also, being completely slope divisible in Zink's sense is a local property.

\begin{proposition}
Let $G=\GL_\rr$. Then a local $\GL_\rr$-shtuka with $K_0$-structure $\ul\CG$ over $S$ is completely slope divisible in the sense of Zink if and only if its associated local shtuka $\ulV=(V,\phi)$ satisfies the following condition:

\noindent
There is a sequence of local quotient-shtukas $\ulV=\ulV_e\onto\ulV_{e-1}\onto\ldots\onto\ulV_1\onto\ulV_0=(0)$ and integers $t_1>\ldots>t_e$ and $s>0$ such that for all $1\le i\le e$
\begin{enumerate}
\item $z^{-t_i}\phi^s\bigl((\sigma^s)^\ast V_i\bigr)\subset V_i$ and
\item on $W_i:=\ker(V_i\onto V_{i-1})$ the map $z^{-t_i}\phi^s:(\sigma^s)^\ast W_i\isoto W_i$ is an isomorphism.
\end{enumerate}
\end{proposition}

Note that the latter condition is indeed the same as complete slope divisibility in \cite[Definition 1.2]{OortZink} when one views local shtukas as analogs of the contravariant Dieudonn\'e modules of $p$-divisible groups. 

\begin{proof}
Assume that $\ul\CG$ is completely slope divisible in the sense of Zink and let $P,M,\mu,s,\ul{\bar\CP}$ be as in Definition~\ref{DefSlopeFilt2}. We may even assume that $M$ is the centralizer of $\mu$ after possibly enlarging $M$ and $P$. We pass to an \'etale covering $S'\to S$ in order to choose trivializations of $\ul{\bar\CP}$ and $\ul\CG$ such that $z^{-\mu}\phi^s=\bar p\s$ with $\bar p\in\ol P\bigl(\Gamma(S',\CO_{S'})\dbl z\dbr\bigr)$. By our choice of $B$ as the upper triangular matrices, $M$ consists of all block diagonal matrices of fixed block sizes $d_1,d_2,\ldots,d_e$ and $\ol P$ consists of the corresponding lower triangular block matrices. The successive projections onto the upper left blocks uniquely defines a sequence of local quotient-shtukas $\ulV_{S'}=\ulV_e\onto\ulV_{e-1}\onto\ldots\onto\ulV_1\onto\ulV_0=(0)$ over $S'$. Namely $V_i$ is spanned by the first $d_1+\ldots+d_i$ basis vectors. By its uniqueness this sequence descends to $S$. The $\GL_\rr$-dominant cocharacter $\mu$ can be written as 
\[
\mu\;=\;(t_1,\ldots,t_1\,,\,t_2,\ldots,t_2\,,\,\ldots\,,\,t_e,\ldots,t_e)
\]
with $t_1>t_2>\ldots>t_e$ since $M$ is the centralizer of $\mu$. Then $z^{-t_i}\phi^s=z^{\mu-t_i}\bar p$ and this implies (a) and (b) because $\mu$ is central in $M$. 

Conversely assume that $\ul\CG$ satisfies the condition of the proposition. Over a Zariski-covering $S'\to S$ we can successively choose compatible trivializations of all $W_i$ and $V_i$ as free $\CO_{S'}\dbl z\dbr$-modules. We order the obtained basis vectors of $V$ such that the first $d_1+\ldots+d_i$ of them project onto a basis of $V_i$, where $d_i=\rk W_i$. With respect to this basis the Frobenius of $\ulV$ is given by an element $\bar p'\in L\ol P(S')$ for $\ol P$ the parabolic consisting of the lower triangular block matrices whose Levi component $M$ contains all block diagonal matrices with block sizes $d_1,\ldots,d_e$. The opposite parabolic $P$ is standard. The local $\ol P$-shtuka with $L\ol P\cap K_0$-structure $\ul{\bar\CP}'=\bigl((L\ol P\cap K_0)_{S'},\,\bar p'\s\bigr)$ over $S'$ descends to a local $\ol P$-shtuka $\ul{\bar\CP}$ over $S$ which satisfies $\ul\CG\cong i_\ast\ul{\bar\CP}$. Let $\mu$ be the $\GL_\rr$-dominant cocharacter with weights $t_i$ of multiplicity $d_i$. Then $\mu:\BG_m\to\T\subset M$ is central in $M$. By (a) the Frobenius $z^{-\mu}\phi^s=z^{-\mu}\bar p'\s(\bar p')\cdot\ldots\cdot(\sigma^{s-1})^\ast(\bar p')$ lies in $\Gamma(S',\CO_{S'})\dbl z\dbr^{\rr\times \rr}$ By (b) its Levi part, that is its semi-simplification, lies in $M\bigl(\Gamma(S',\CO_{S'})\dbl z\dbr\bigr)$. Altogether this implies that $z^{-\mu}\phi^s\in\ol P\bigl(\Gamma(S',\CO_{S'})\dbl z\dbr\bigr)$ and the proposition is proved.
\end{proof}

\begin{lemma}
A local $G$-shtuka with $K_0$-structure $\ul\CG$ over $S$ which is completely slope divisible (in the sense of Definition \ref{DefSlopeFilt1}) is completely slope divisible in the sense of Zink.
\end{lemma}

\begin{proof}
As both conditions are local on $S$ we may replace $S$ by an \'etale covering and assume that $\CG$ is trivializable. Since $\ul\CG$ is completely slope divisible we may choose a trivialization $\ul\CG\cong(K_{0,S},b\s)$ with $b=xm\bar n$ and $m\in I_M(S),\,\bar n\in I_{\ol N}(S)$. Then
\[
\phi^s\;=\;x^s\cdot x^{-(s-1)}m\bar nx^{s-1}\cdot\ldots\cdot x^{-1}(\sigma^{s-2})^\ast(m\bar n)x\cdot (\sigma^{s-1})^\ast(m\bar n)\;=:\;x^s\cdot\bar p\,.
\]
Since $x$ is $P$-fundamental we obtain $\bar p\in I_M(S)I_{\ol N}(S)\subset \ol P\bigl(\Gamma(S,\CO_{S})\dbl z\dbr\bigr)$. Taking $s>0$ such that $\mu:=x^s\in X_\ast(\T)\subset\wt W$ the assertion follows.
\end{proof}

%%%%%%%%%%%%%%%%%%%%%%%%%%%%%%%%%%%%%%%%%%%%%%%%%%%%%%%%%%%%%%%%%%%%%%
%
%   The central leaf
%
%%%%%%%%%%%%%%%%%%%%%%%%%%%%%%%%%%%%%%%%%%%%%%%%%%%%%%%%%%%%%%%%%%%%%%

\section{Central leaves and the product structure on Newton strata}\label{secadlv}

\begin{definition}\label{DefCentralLeaf}
Let $\BF/\BF_q$ be a field extension and let $\ul\BG$ be a local $G$-shtuka with $K$-structure over $\BF$. Let $S$ be a noetherian $\BF$-scheme and let $\ul\CG$ be a bounded local $G$-shtuka with $ K$-structure over $S$. Consider the subset
$$
\CC_{\ul\BG,S}=\bigl\{s\in S: \ul\BG_{k}\cong\ul\CG_s\otimes_{k(s)}k \text{ over an algebraically closed field } k/k(s)\bigr\}.$$
By Corollary~\ref{ThmCentralLeaf} below it is a locally closed subscheme. Any geometrically irreducible component of $\CC_{\ul\BG,S}$ with the induced reduced subscheme structure is called a \emph{central leaf corresponding to $\ul\BG$ in $S$}. Note that possibly  the central leaves are only defined over a finite extension of $\BF$.
\end{definition}

\begin{lemma}\label{LemmaCLeafConstructible}
Let $\BF\supset\BF_q$ be a field and let $S$ be a reduced $\BF$-scheme. Let $\ul\BG=\bigl( K_\BF,b\s\bigr)$ and $\ul\CG=\bigl( K_S,g\s\bigr)$ be local $G$-shtukas with $K$-structure over $\BF$, respectively over $S$. Then the central leaf $\CC_{\ul\BG,S}$ is constructible in $S$. In particular it contains a subset $U\subset\CC_{\ul\BG,S}$ which is open and dense in the closure of $\CC_{\ul\BG,S}$.
\end{lemma}
Before proving this Lemma we provide two boundedness results.
\begin{lemma}\label{LemmaConjBd}
Let $\mu\in X_\ast(\T)$ be dominant. Then there is a positive integer $e=e(\mu)\in\BN$ such that for every $\BF_q$-scheme $S$ and for every $h\in LG(S)$ which is bounded by $\mu$
\[
hI_{n+e}(S)h^{-1}\;\subset\;I_n(S)\qquad\text{for all }n\,.
\]
\end{lemma}

\begin{proof}
It is enough to consider the case that $h$ corresponds to the generic point of $K_0 z^{\mu}K_0$ (as every $h$ as in the assertion lies in the closure of that point). Especially $S=\Spec k'$ for a field $k'$. We decompose $h$ as $g_1z^{\mu}g_2$ for $g_1,g_2\in K_0(k')$. As $K_i$ is a normal subgroup of $K_0$ for each $i$ we have $hI_{n+e}(S)h^{-1}\subset hK_{n+e}(S)h^{-1}=g_1z^{\mu}K_{n+e}(S)z^{-\mu}g_1^{-1}$. Choosing $e$ so large that $z^{\mu}K_{n+e}z^{-\mu}\subset K_{n+1}\subset I_{n}$  the lemma follows.
\end{proof}

We call a subset of $LG(k)=G\bigl(k\dpl z\dpr\bigr)$ \emph{bounded} if it is contained in a finite union of $K_0(k)$-double cosets $K_0(k)z^\mu K_0(k)$ for dominant $\mu\in X_\ast(\T)$. 

\begin{lemma}\label{ThmHV9.1}
Let $\CB$ be a bounded subset of $LG(k)$ for an algebraically closed field $k$ and let $K$ be an open subgroup of $LG$. Then there is a constant $c=c(\CB,K)\in \BN$ such that for each $b\in\CB$ and each $h\in I_{c}(k)$ there is an $f\in K(k)$ with $bh = f^{-1}b\s(f)$.
\end{lemma}

\begin{proof}
In \cite[Theorem 10.1]{HV1} a slightly stronger statement was proved for the case $K=I_n$ for some $n$. As the subgroups $I_n$ are cofinal in the set of open subgroups of $LG$, the lemma follows.
\end{proof}

\begin{proof}[Proof of Lemma \ref{LemmaCLeafConstructible}]
Let $\mu$ be the Hodge point of $b^{-1}$ and let $e=e(\mu)\in\BN$ be the constant from Lemma~\ref{LemmaConjBd}. Also let $\CB=\{b\}$ and let $c=c(\CB,K)$ be the constant from Lemma~\ref{ThmHV9.1}. We consider the functor on the category of $S$-schemes $X$
\[
Y(X) \;=\; \bigl\{\,f:X\to K/I_{c+e}:\;b_X^{-1}f^{-1}g_X\s(f)\in I_{c}(X)\,\bigr\} 
\]
Here the condition $b_X^{-1}f^{-1}g_X\s(f)\in I_{c}(X)$ means the following. Over an \'etale covering $\wt X$ of $X$ the morphism $f$ is represented by an element $f\in K(\wt X)$. Then the condition is independent of the covering $\wt X$ and the choice of the element $f$. Indeed if $f=\tilde f h$ with $h\in I_{c+e}(\wt X)$ then 
\[
b_{\wt X}^{-1}f^{-1}g_{\wt X}\s(f)\;=\; (b_{\wt X}^{-1}h^{-1}b_{\wt X})b_{\wt X}^{-1}\tilde f^{-1}g_{\wt X}\s(\tilde f)\s(h)\,.
\]
By Lemma~\ref{LemmaConjBd} the terms $b_{\wt X}^{-1}h^{-1}b_{\wt X}$ and $\s(h)$ lie in $I_{c}(\wt X)$. This shows that $Y$ is well defined and is representable by a scheme of finite type over $S$ which is a closed subscheme of $(K/I_{c+e})\times_{\BF_q}S$. We claim that $\CC_{\ul\BG,S}$ is the image of the morphism $Y\to S$. From this the lemma follows by Chevalley's theorem~\cite[$\mathrm{IV}_1$, Corollaire 1.8.5]{EGA} and \cite[$0_3$, Proposition 9.2.2]{EGA}.

To prove the claim let $s\in \CC_{\ul\BG,S}$, let $k$ be an algebraically closed field, and let $f_s\in K(k)$ be an isomorphism $\ul\BG_{k}\isoto\ul\CG_s\otimes_{k(s)}k$, that is $f_s^{-1}g_s\s(f_s)=b_s$. Then in particular $f_s\in Y(k)$ is a point mapping to $s\in\CC_{\ul\BG,S}$. Conversely let $s\in S$ lie in the image of a point $f_s\in Y(k)$ for an algebraically closed field $k$, that is $f_s\in K(k)$ with $b_s^{-1}f_s^{-1}g_s\s(f_s)\in I_{c}(k)$. Then Lemma \ref{ThmHV9.1} yields an $\el\in K(k)$ with $f_s^{-1}g_s\s(f_s)=\el^{-1}b_s\s(\el)$. Thus $f_s\el^{-1}$ is an isomorphism $\ul\BG_s\isoto\ul\CG_s$ and so $s$ belongs to $\CC_{\ul\BG,S}$.
\end{proof}

To state the result on the product structure on Newton strata we need a generalization to more general $K$ of affine Deligne-Lusztig varieties. Let $b\in LG(k)$. Let $\mu\in X_*(\T)$ be dominant. Then the affine Deligne-Lusztig variety associated with $(b,\mu,K)$ is the reduced subscheme $X_{\mu,K}(b)$ of $LG/K$ defined by
\begin{equation}\label{defgenadlv}
X_{\mu,K}(b)(k)=\{g\in LG/K(k): g^{-1}b\s (g)\in K_0(k)z^{\mu}K_0(k)\},
\end{equation}
i.e.~it is simply the reduced inverse image of $X_{\mu}(b)=X_{\mu,K_0}(b)$ under $LG/K\rightarrow LG/K_0$. The closed affine Deligne-Lusztig variety is $X_{\preceq\mu,K}(b)=\bigcup_{\mu'\preceq\mu}X_{\mu,K}(b)$. Note that one could replace the condition in \eqref{defgenadlv} by any condition of the form $g^{-1}b\s (g)\in \CB$ for some bounded subset $\CB$ of $LG$ which a union of $K$-double cosets. In this way one obtains finer notions of affine Deligne-Lusztig varieties. Of particular interest is the case $K=I$ where the set of double cosets is parametrized by $\wt W$. For $y\in \wt W$, the affine Deligne-Lusztig variety $X_y(b)$ associated with $(y,b)$ is defined as the reduced subscheme of the affine flag variety with
\begin{equation}\label{glxx(b)}
X_{y}(b)(k)=\{g\in LG/I(k): g^{-1}b\s (g)\in I(k)yI(k)\}.
\end{equation}

\begin{theorem}\label{ThmProductStructure}
Let $\ul{\BG}=\bigl(K_k,b\s\bigr)$ be a local $G$-shtuka with $K$-structure over an algebraically closed field $k$ bounded by a dominant $\mu\in X_*(\T)$ with Newton point $\nu$. Let $\CN_\nu$ be the Newton stratum of $[b]$ in the universal deformation space of $\ul{\BG}$ bounded by $\mu$ (Proposition~\ref{PropDefo}). Let $x$ be a $P$-fundamental alcove associated with $[b]$ where the Levi subgroup $M$ of $P$ equals the centralizer of the $M$-dominant Newton point of $x$. Then there is a reduced scheme $S$ and a finite surjective morphism $S\twoheadrightarrow\CN_{\nu}$ which factors into finite surjective morphisms $S\overset{\beta}{\rightarrow} X_{\preceq\mu,K}(b)^\wedge\,\whtimes_k\, (I_{\ol N}/x^{-1}I_{\ol N}x)^{\wedge}\overset{\alpha}{\rightarrow} \mathcal{N}_{\nu}$. Here $X_{\preceq\mu,K}(b)^\wedge$ and $(I_{\ol N}/x^{-1}I_{\ol N}x)^{\wedge}$ denote the completions of $X_{\preceq\mu,K}(b)$, respectively $I_{\ol N}/x^{-1}I_{\ol N}x$ at $1$. Furthermore, $\CC_{\ul{\BG},\mathcal{N}_{\nu}}$ is geometrically irreducible and equal to the image of $\{1\}\whtimes_k(I_{\ol N}/x^{-1}I_{\ol N}x)^{\wedge}$ in $\CN_{\nu}$.
\end{theorem}
In particular, $X_{\preceq\mu,K}(b)^\wedge\,\whtimes_k\,(I_{\ol N}/x^{-1}I_{\ol N}x)^{\wedge}$ and $\mathcal{N}_{\nu}$ are isomorphic up to finite surjective morphisms and there is a unique central leaf containing the special point of $\mathcal{N}_{\nu}$. Oort~\cite{Oort04} proves a similar product decomposition of the Newton strata in the deformation space of a $p$-divisible group and in the Siegel moduli space. Theorem \ref{ThmProductStructure} is a generalization to arbitrary split reductive groups and level subgroups $K$ of the translation of his result to the function field case.

Note that we do not claim that $S$ or the morphisms occurring in the theorem are canonical. They will depend on various choices made in the course of the proof.
\begin{proof}
{\it Step 1.} Let $g\in LG(k)$ with $g^{-1}x\s(g)=b$. In order to make the proof easier to digest let us explain its structure. The key observation is the following. Let $T$ be a $k$-scheme, let $n$ be a positive integer and consider $T$-valued points $h\in LG(T)$ and $\delta\in x^{-n}I_{\ol N}(T)x^n\subset I(T)$ such that the induced morphism $h:T\to LG/K$ factors through $X_{\preceq\mu,K}(b)$. Then we can construct out of $(h,\delta)$ the local $G$-shtuka $\bigl(K_T,(gh)^{-1}x\delta\s(gh)\bigr)$ which has two desired properties. Firstly, it has constant Newton point $\nu$ since all elements of $xI(T)$ have constant Newton point $\nu$ by Remark \ref{remfundalc} \ref{e}. Secondly, from $h\in X_{\preceq\mu,K}(b)(T)$ we obtain that $(gh)^{-1}x\s(gh)=h^{-1}b\s(h)$ is bounded by $\mu$. If we choose $n$ big enough, we may hope that $\s(gh)^{-1}\delta \s(gh)\in K_0(T)$. Then $\bigl(K_T,(gh)^{-1}x\delta\s(gh)\bigr)$ will be bounded by $\mu$. This already indicates how we will construct the morphism $\alpha: X_{\preceq\mu,K}(b)^\wedge\,\whtimes_k\, (I_{\ol N}/x^{-1}I_{\ol N}x)^{\wedge}\rightarrow \mathcal{N}_{\nu}$. The proof now proceeds in several steps. We construct $S$ in step 2, estimate the $n$ needed for the above in step 3, define the morphism $\alpha$ in step 4, the morphism $\beta:S\to X_{\preceq\mu,K}(b)^\wedge\,\whtimes_k\, (I_{\ol N}/x^{-1}I_{\ol N}x)^{\wedge}$ in steps 5 and 6, prove the finiteness and surjectivity of $\alpha$ and $\beta$ in step 7, and compute the central leaf $\CC_{\ul{\BG},\mathcal{N}_{\nu}}$ in step 8. See the beginning of each step for more details.

\smallskip%\noindent
{\it Step 2. Construction of a first finite surjective morphism $S\to\CN_\nu$ over which the universal local $G$-shtuka is isogenous to a completely slope divisible one, and computation of two bounds $\mu'$ and $\mu''$.} 
Consider the quasi-isogeny groups $J_b$ of $\ul\BG$ and $J_x$ of $(K_k,x\s)$. Then $gJ_bg^{-1}=J_x\subset M\bigl(k\dpl z\dpr\bigr)$ for the Levi subgroup $M$ by Corollary~\ref{PropQISGpDecent}. Let $\mu'$ be such that all elements of $gj X_{\preceq\mu,K}(b)^{\wedge}$ are bounded by $\mu'$ for all $j\in J_b\cap K$. Such a $\mu'$ exists by \cite[Lemma 3.13]{HV1} on a quasi-compact connected neighborhood of $1\in X_{\preceq\mu,K}(b)$ since $X_{\preceq\mu,K}(b)$ is reduced, all $gj$ have the same image in $\pi_1(G)$, and one only has to consider a finite set of representatives $j$ for $(J_b\cap K)/(J_b\cap K\cap g^{-1}Kg)$. Now let $\ul\CG$ be the universal local $G$-shtuka over $\CN_{\nu}$. We consider the normalization $\wt\CN_{\nu}\rightarrow \CN_{\nu}$ of $\CN_{\nu}$. Since $\CN_\nu$ is excellent this is a finite morphism by \cite[IV, Scholie 7.8.3(iii) and (vi)]{EGA}. Hence $\wt\CN_\nu$ is normal and its connected components are integral by \cite[II, Corollary 6.3.8]{EGA}.
Applying Corollary \ref{CorSlopeFilt} to each such component we see that there is a finite surjective morphism $S=\Spec R_S\rightarrow\CN_{\nu}$ with $R_S$ a direct sum of integral domains, such that over $S$ the local $G$-shtuka $\ul\CG$ is isomorphic to $(K_S,l^{-1}x\epsilon\s(l)\s)$ for some $l\in LG(S)$ and some $\epsilon\in I_{\ol N}(S)$ whose reductions in the finitely many maximal ideals of $R_S$ are $g$ and $1$, respectively. Since $S$ is quasi-compact and reduced, and the image of $l$ in $\pi_1(G)$ is trivial, $l$ is bounded by some dominant $\mu''$ according to \cite[Lemma 3.13]{HV1}.

\smallskip%\noindent
{\it Step 3. Computation of the bound on $n$ needed in step 1, first replacement of $S$, and choice of a section $(I_{\ol N}/x^{-1}I_{\ol N}x)^{\wedge}\cong\CI_{x,n}\hookrightarrow x^{-n}I_{\ol N}x^n$.}
Let $n_K\in \BN$ with $I_{n_K}\subset K$. Let $n\in \BN$ with $h^{-1}x^{-n}I_{\ol N}x^nh\subset I_{n_K}$ for all $h\in LG$ bounded by $\mu'$ or by $\mu''$ (which exists by Remark \ref{remfundalc} \ref{d} and Lemma \ref{LemmaConjBd}). By Lemma~\ref{LemmaBoundedN} there exists an $n_1\ge n$ with $\s(\tilde j)^{-1}x^{-n_1}I_{\ol N}x^{n_1}\s(\tilde j)\subset x^{-n}I_{\ol N}x^n$ for all $\tilde j$ contained in the bounded set $g(J_b\cap K)g^{-1}\subset M\bigl(k\dpl z\dpr\bigr)$. By Lemma \ref{LemmaTrivN} we can replace $S$ by a finite extension and multiply $l$ by an element of $I_{\ol N}(S)$ to achieve that $\epsilon$ as above is even in $x^{-n_1}I_{\ol N}(S)x^{n_1}\subset x^{-n}I_{\ol N}(S)x^n$. Let $\CI_{x,n}$ be the completion of $(x^{-n}I_{\ol N}x^n)/(x^{-(n+1)}I_{\ol N}x^{n+1})$ at $1$. Then $\CI_{x,n}\cong (I_{\ol N}/x^{-1}I_{\ol N}x)^{\wedge}$. We choose a section $\CI_{x,n}\rightarrow x^{-n}I_{\ol N}x^{n}$ using Lemma \ref{lemsecunip} and from now on identify $\CI_{x,n}$ with its image. Note that as $I_{\ol N}$ is pro-unipotent we may lift $\{1\}\rightarrow x^{-(n+1)}I_{\ol N}x^{n+1}$ (where $1$ is the special point of $\CI_{x,n}$ and the map is the restriction of the section) to a morphism $\CI_{x,n}\rightarrow x^{-(n+1)}I_{\ol N}x^{n+1}$. Multiplying the original section by the inverse of this map, we see that we may assume that the section maps the special point of $\CI_{x,n}$ to $1\in x^{-n}I_{\ol N}x^{n}$. 

\smallskip%\noindent
{\it Step 4. Definition of $\alpha:X_{\preceq\mu,K}(b)^\wedge\,\whtimes_k\,\CI_{x,n}=:\wh T\rightarrow \CN_{\nu}$.} We define $\alpha(h,\delta):= (gh)^{-1}x\delta\s(gh)$ where $h\in LG(\wh T)$ is a chosen, fixed representative (existing by Lemma~\ref{LemmaTildeKEtale}) of the universal element of $X_{\preceq\mu,K}(b)(\wh T)$ and where $\delta\in x^{-n}I_{\ol N}(\wh T)x^{n}$ is the image of the universal element. We consider the local $G$-shtuka $(K_{\wh T},(gh)^{-1}x\delta\s(gh)\s)$ over $\wh T$ together with the isomorphism at its closed point $h_k:(K_k,(gh)^{-1}x\s(gh)\s)\isoto\ul\BG$. It yields a deformation of $\ul\BG$ over $\wh T$ which is independent of the representative $h$ up to canonical isomorphism. By the observation in step 1, this deformation has constant Newton point $\nu$ and is bounded by $\mu$ because the choice of $n$ implies that $\s(gh)^{-1}\delta \s(gh)\in K(\wh T)\subset K_0(\wh T)$. This indeed defines $\alpha:\wh T\to\CN_\nu$.

\smallskip%\noindent
{\it Step 5. Final definition of $S$ and construction of a morphism $S\to X_{\preceq\mu,K}(b)^\wedge$, i.e. of the first component of $\beta$.} We want to use the description of $\ul\CG_S$ by $l$ and $\epsilon$ to construct a finite surjective morphism $\beta:S\rightarrow X_{\preceq\mu,K}(b)^\wedge\,\whtimes_k\,\CI_{x,n}$. We first define the map $S\to X_{\preceq\mu,K}(b)^\wedge$. For this purpose we choose an $r\in\BN$ with $K_r\subset K\cap hKh^{-1}$ for all $h\in X_{\preceq\mu,K}(b)^\wedge$ (use Lemma \ref{LemmaConjBd}). In order to obtain the surjectivity of $\beta$ in step 8 below we rename our $S$ to $S_1$ and replace it by the finite disjoint union $\coprod_{j\in\Gamma}S_j$ with  $S_j=S_1$ for all $j\in\Gamma=(J_b\cap K)/(J_b\cap K_r)$. The first component of the morphism $\beta$ is given on $S_j$ by $h=j^{-1}g^{-1}l\in(LG/K)(S_j)$. To show that $h\in X_{\preceq\mu,K}(b)^{\wedge}(S_j)$ we consider it in every geometric generic point $\ol\eta$ of $S_j$. By Corollary~\ref{Cor4.10'} there is a $y\in x^{-n}I_{\ol N}(\ol\eta)x^n$ with $x\epsilon=y^{-1}x\s(y)$. Together with the choice of $\mu''$ and $n$ this implies that $h^{-1}b\s(h)=\kkk(l^{-1}x\epsilon\s(l))\s(\kkk)^{-1}$ with $\kkk=l^{-1}yl=(gjh)^{-1}ygjh\in K(\bar\eta)$. Hence $h_{\ol\eta}\in X_{\preceq\mu,K}(b)(\bar\eta)$. Since $X_{\preceq\mu,K}(b)\subset LG/K$ is closed, $S_j$ is reduced and complete, and the reduction of $h$ in all closed points of $S_j$ is $1\in LG/K$ we see that $h$ gives a morphism $S_j\to X_{\preceq\mu,K}(b)^{\wedge}$.

\smallskip%\noindent
{\it Step 6. Definition of the second component $S\to\CI_{x,n}$ of $\beta$.} On $S_j$ we would like to choose $\delta\in\CI_{x,n}(S_j)$ equal to $\tilde\epsilon:=\s(gj^{-1}g^{-1})\epsilon\s(gjg^{-1})\in x^{-n}I_{\ol N}(S_j)x^n$ because then $(gh)^{-1}x\tilde\epsilon\s(gh)=l^{-1}x\epsilon\s(l)$ and $\alpha\beta$ equals the projection $S_j\to\CN_\nu$. However, $\tilde\epsilon$ is in general not in the image $\CI_{x,n}$ of the chosen section, and needs first to be modified by $\sigma$-conjugating $(gh)^{-1}x\tilde\epsilon\s(gh)$ by further elements of $K$. To do that and to define the second component $\delta$ we use induction on $m$ to define it modulo $(\s)^m(\mathfrak{m})$ where $\mathfrak{m}$ denotes the maximal ideal of any component of $R_{S_j}$. More precisely, we want to show that for every $m\in\BN$ there is a $g_m\in (gh)^{-1}x^{-n}I_{\ol N}(S_j)x^ngh$ such that $g_m^{-1}(gh)^{-1}x\tilde\epsilon\s(gh g_m)=(gh)^{-1}x\epsilon_m\delta_m\s(gh)$ for some $\epsilon_m\in x^{-(n+1)}I_{\ol N}(S_j)x^{n+1}$ which is congruent to $1$ modulo $(\s)^m(\mathfrak{m})$ and some $\delta_m\in\CI_{x,n}(S_j)$. In addition we choose $g_m$ and $\delta_m$ in such a way that $g_{m+1}\equiv g_m$ and $\delta_m\equiv \delta_{m-1}$ modulo $(\s)^m(\mathfrak{m})$.  Note that our choice of $\mu'$ and $n$ implies that $g_m\in (gh)^{-1}x^{-n}I_{\ol N}x^ngh\subset K$. Hence this iterative process of $\sigma$-conjugating $(gh)^{-1}x\tilde\epsilon\s(gh)$ only changes the trivialization of the local $G$-shtuka with $K$-structure. For $m=0$ let $g_0=1$, and let $\delta_0\in \CI_{x,n}(S_j)$ be the unique element with $\epsilon_0=\tilde\epsilon\delta_0^{-1}\in x^{-(n+1)}I_{\ol N}(S_j)x^{n+1}$. That is, $\delta_0^{-1}$ is the image of $\tilde\epsilon^{-1}$ under the map
\begin{equation}\label{Eq7.1}
x^{-n}I_{\ol N}x^n\;\to\;x^{-n}I_{\ol N}x^n/x^{-(n+1)}I_{\ol N}x^{n+1}\;\to\;\CI_{x,n}.
\end{equation} 
For the induction step we set $g_{m+1}=g_m(gh)^{-1}x\epsilon_mx^{-1}gh$. Then
\begin{eqnarray*}
g_{m+1}^{-1}(gh)^{-1}x\tilde\epsilon\s(gh g_{m+1})&=&((gh)^{-1}x\epsilon_mx^{-1}gh)^{-1}(gh)^{-1}x\epsilon_m\delta_m\s(gh)\s((gh)^{-1}x\epsilon_mx^{-1}gh)\\
&=&(gh)^{-1}x\delta_m\s(x\epsilon_mx^{-1})\s(gh).
\end{eqnarray*}
Note that $\s(x\epsilon_mx^{-1})\in x^{-n}I_{\ol N}(S_j)x^n$ is congruent to 1 modulo $(\s)^{m+1}(\mathfrak{m})$. We now decompose $\delta_m\s(x\epsilon_mx^{-1})$ as $\epsilon_{m+1}\delta_{m+1}$ with $\delta_{m+1}\in \CI_{x,n}(S_j)$ and $\epsilon_{m+1}\in x^{-(n+1)}I_{\ol N}(S_j)x^{n+1}$ by taking $\delta_{m+1}^{-1}$ as the image of $\s(x\epsilon_m^{-1}x^{-1})\delta_m^{-1}$ under the map (\ref{Eq7.1}). Then $\delta_{m+1}\equiv\delta_m\pmod{(\s)^{m+1}(\mathfrak{m})}$ and $\epsilon_{m+1}\equiv 1\pmod{(\s)^{m+1}(\mathfrak{m})}$. We define $\delta$ and $\hat g$ to be the limits of the $\delta_m$, respectively $g_m$ and replace $l$ by $l\hat g$ and $h$ by $h\hat g$. Note that this does not change the point $h\in X_{\preceq\mu}(b)(S_j)^{\wedge}$. Then the Frobenius of the universal local $G$-shtuka over $S_j$ is $\hat g^{-1}l^{-1}x\epsilon\s(l\hat g)=(gh\hat g)^{-1}x\tilde\epsilon\s(gh\hat g)=(gh)^{-1}x\delta\s(gh)$. This defines the second component $\delta$ of $\beta$ on $S_j$ and we put everything together to obtain the morphism $\beta:S=\coprod_{j\in\Gamma}S_j\to X_{\preceq\mu,K}(b)^\wedge\,\whtimes_k\,\CI_{x,n}$, such that the projection $S\rightarrow \CN_{\nu}$ factors through the morphisms $\beta$ and $\alpha:X_{\preceq\mu,K}(b)^\wedge\,\whtimes_k\,\CI_{x,n}\rightarrow \CN_{\nu}$.

\smallskip%\noindent
{\it Step 7. Proof that $\alpha$ and $\beta$ are finite surjective.} Clearly, the morphism $\beta$ is finite and $\alpha$ is surjective because $S\to\CN_\nu$ is. To show that also $\alpha$ is finite and $\beta$ is surjective we write $\CN_\nu=\Spec A$, and $X_{\preceq\mu,K}(b)^\wedge\,\whtimes_k\,\CI_{x,n}=\Spec B$, and $S=\Spec C$. Since $A$ is noetherian by Proposition~\ref{PropDefo}, it suffices for the finiteness of $\alpha$ to show that the $A$-module $B$ is via $\beta^\ast$ contained in the finite $A$-module $C$. Let $\Fb=\ker(B\to C)$. Since $B$ is reduced we only need to show that $\Fb$ is contained in every minimal prime ideal $\Fp$ of $B$. We will show that $\Fp$ lies in the image of $\beta$. Then $\Fb\subset\Fp$. Also since $\beta$ is finite, hence closed, this will prove that $\beta$ is surjective. Consider the morphism $\Spec C\otimes_AB/\Fp\to\Spec B/\Fp$ which is finite surjective by base change from $\Spec C\to\Spec A$. Let $\FP\subset C\otimes_AB/\Fp$ be a prime ideal mapping to $\Fp$ and set $R=(C\otimes_AB/\Fp)/\FP$. Then $R$ is a complete local integral domain. It comes with two projections $\Spec R\to\Spec B$, the first of which factors through $\Spec C$, both mapping the maximal ideal of $R$ to the closed point of $\Spec B$. The projections yield two $R$-valued points $(h,\delta)$ and $(h',\delta')$ of $X_{\preceq\mu,K}(b)^\wedge\,\whtimes_k\,\CI_{x,n}$ with reduction $(1,1)$, which have the same image in $\CN_\nu(R)$, i.e.~$(gh)^{-1}x\delta\s(gh)=\tilde l^{-1}(gh')^{-1}x\delta'\s(gh')\s(\tilde l)$ for the representatives $h,h'\in LG(R)$ induced by the choice of representative in the definition of $\alpha$, and for some $\tilde l\in K(R)$ such that $h'\tilde lh^{-1}$ has reduction $1$. By construction the point $(h,\delta)$ lies in the image of $\beta$, and we have to show that $(h',\delta')$ likewise does. 

To compare $h,h'$ we consider them in each geometric point of $\Spec R$. We use again that over a perfect field $x\delta=y^{-1}x\s(y)$ and $x\delta'=(y')^{-1}x\s(y')$ for some $y,y'\in x^{-n}I_{\ol N}x^n$. Together with the choice of $n$ this implies that $h^{-1}b\s(h)= \kkk^{-1}(h')^{-1}b\s(h')\s(\kkk)$ for $\kkk=(gh')^{-1}y'(gh')\tilde l(gh)^{-1}y^{-1}(gh)\in K$, that is $h'\tilde fh^{-1}\in J_b$. Let $\xi=h'h^{-1}$. The computation just made shows that in every geometric point, $\xi\in J_b\cdot hKh^{-1}$. We consider the image of $\xi$ in $LG/(hKh^{-1})(R)$, which lies in the image of $J_b$ in this quotient (as this is the case in each geometric point and $R$ is reduced). The latter is a discrete subscheme, its intersection with each bounded subscheme of $LG/(hKh^{-1})$ is finite. Hence the image of $\xi$ is constant and equal to the image of some $j\in J_b$. Thus $j^{-1}\xi\in hK(R)h^{-1}$ and $h'=j h$ as elements of $LG/K(R)$. Considering the reduction in the special point of $R$ we see that $j\in K$. Furthermore $j$ is unique up to right multiplication by elements of $J_b\cap K\cap hKh^{-1}$. So by our choice of $S=\coprod_{j\in\Gamma}S_j$ with $S_j=S_1$ we can consider the $R$-valued point $\Spec R\to S$ which lies in some $S_{j_0}$ also as a point in $S_{jj_0}$. As such it is mapped under $\beta$ to a point $(jh,\tilde\delta)=(h',\tilde\delta)$ in $X_{\preceq\mu,K}(b)^\wedge\,\whtimes_k\,\CI_{x,n}$.

It thus remains to show that for $h=h'$ in $LG/K(R)$, also $\delta=\delta'$ in $\CI_{x,n}$. Then $(h',\delta')$ lies in the image of $\beta$ as desired. Let $\mathfrak{m}_ R$ be the maximal ideal of $R$. By definition $\delta,\delta'\in x^{-n}I_{\ol N}x^n(R)$. Let $m$ be minimal with $\delta'\delta^{-1}\notin x^{-(n+1)}I_{\ol N}(R/(\s)^m(\mathfrak{m}_R))x^{n+1}$. As $\delta,\delta'$ are both in the image of the section of $\CI_{x,n}$ this implies $\delta'\delta^{-1}=1\in LG(R/(\s)^{m-1}(\mathfrak{m}_R))$. Our assumption that the deformations associated with $(h,\delta)$ and $(h,\delta')$ are isomorphic as deformations (i.e.~via an isomorphism which is the identity on the special fiber) implies that there is an $f\in LG(R)$ with $f\equiv 1\pmod{\mathfrak{m}_R}$ and $f^{-1}x\delta\s(f)=x\delta'$. This is equivalent to $\delta^{-1}x^{-1}fx\delta'=\s(f)$. Considering the equation for $f$ modulo $(\s)^l(\mathfrak{m}_R)$ for all $l\in\BN$ we see that $f\in \ol N\bigl(R\dpl z\dpr\bigr)$ since $x$ normalizes $\ol N$. Considering the equation modulo $x^{-i}I_{\ol N}x^i$ for the maximal $i$ with $f\in x^{-i}I_{\ol N}x^i(R)$ we see that this $i$ is at least $n$, especially $x^{-1}f^{-1}x\in x^{-(n+1)}I_{\ol N}x^{n+1}$. Furthermore $f=x\delta\s(f)(\delta')^{-1}x^{-1}\equiv 1 \pmod{(\s)^{m-1}(\mathfrak{m}_R)}$ by induction, thus $\s(f)\equiv 1 \pmod{(\s)^{m}(\mathfrak{m}_R)}$. Hence over $R/(\s)^m(\Fm_R)$ we have $\delta'\delta^{-1}=\delta'\s(f)^{-1}\delta^{-1}=x^{-1}f^{-1}x\in x^{-(n+1)}I_{\ol N}(R/(\s)^m(\mathfrak{m}_R))x^{n+1}$. But this is a contradiction to our assumption $\delta'\delta^{-1}\notin x^{-(n+1)}I_{\ol N}(R/(\s)^m(\mathfrak{m}_R))x^{n+1}$. This proves $\delta=\delta'$ from which the finiteness of $\alpha$, the surjectivity of $\beta$, and hence the first assertion of the theorem follow. 

\smallskip%\noindent
{\it Step 8. Computation of the central leaves.} To compute $\CC_{\ul\BG,\CN_\nu}$ consider two geometric points of $\CN_\nu$ and some preimages $(h,\delta)$ and $(h',\delta')$ in $X_{\preceq\mu,K}(b)^\wedge\,\whtimes_k\,\CI_{x,n}$. The preceding calculation shows that the local $G$-shtukas with $K$-structure above these points can only belong to the same isomorphism class if $h'=jh$ for some $j\in J_b$. We are interested in the special case $h'=1$. We use that over a perfect field each element $x\delta$ or $x\delta'$ with $\delta,\delta'\in x^{-n}I_{\ol N}x^n$ is $\sigma$-conjugate by an element of $x^{-n}I_{\ol N}x^n$ to $x$. Thus the set of points $(h,\delta)\in X_{\preceq\mu,K}(b)^\wedge\,\whtimes_k\,\CI_{x,n}$ such that the local $G$-shtuka with $K$-structure corresponding to the image in $\CN_{\nu}$ is isomorphic to $\ul\BG$ consists of all pairs $(j,\delta)$ for some $j\in (J_b/J_b\cap K)\cap X_{\preceq\mu,K}(b)^\wedge$ and arbitrary $\delta$. As $j\in X_{\preceq\mu,K}(b)^\wedge$, it is bounded, and hence only a finite set of elements of $J_b/J_b\cap K$ can occur. Furthermore, an element of this finite set of closed points of $LG/K$ only lies in $X_{\preceq\mu,K}(b)^\wedge$ if it is equal to 1. Hence $j=1$ and we see that $\CC_{\ul{\BG},\mathcal{N}_{\nu}}$ is of the form claimed in the theorem.
\end{proof}

An analogous proof shows
\begin{theorem}\label{ThmProductI}
Let $\ul{\BG}=\bigl(I_k,b\s\bigr)$ be a local $G$-shtuka with $I$-structure over $k$ which is of affine Weyl type $y$ and has Newton point $\nu$. Let $\CN_\nu$ be the Newton stratum of $b$ in the universal deformation space of $\ul{\BG}$ of Weyl type $y$. Let $x$ be a $P$-fundamental alcove associated with the isogeny class of $\ul{\BG}$ where $P$ is such that the Levi subgroup $M$ of $P$ equals the centralizer of the $M$-dominant Newton point of $x$. Then up to finite surjective morphisms, $\CN_{\nu}$ is isomorphic to the completion $X_{y}(b)^\wedge\whtimes_k (I_{\ol N}/x^{-1}I_{\ol N}x)^{\wedge}$ of $X_{y}(b)\times_k I_{\ol N}/x^{-1}I_{\ol N}x$ at $(1,1)$. Furthermore, $\CC_{\ul{\BG},\mathcal{N}_{\nu}}$ is geometrically irreducible and equal to the image of $\{1\}\whtimes_k(I_{\ol N}/x^{-1}I_{\ol N}x)^{\wedge}$ in $\CN_{\nu}$.\qed
\end{theorem}

\begin{corollary}\label{ThmCentralLeaf}
Using the notation of Definition~\ref{DefCentralLeaf} let $\nu$ be the Newton point of $\ul\BG$ and let $\CN_{\nu,S}$ be the Newton stratum of $\nu$ in $S$. Then $\CC_{\ul\BG,S}$ is closed in $\CN_{\nu,S}$. In particular $\CC_{\ul\BG,S}$ with the reduced structure is a locally closed subscheme of $S$.
\end{corollary}
\begin{proof}
We write $\ul\BG=(K_{k},b_0\s)$ for some $b_0$ and let $\mu$ be such that $b_0$ is bounded by $\mu$. By Lemma \ref{LemmaCLeafConstructible} $\CC_{\ul\BG,S}$ is constructible, so it is enough to show that it is stable under specialization in $\mathcal{N}_{\nu, S}$. Let $\ul \CF=(K_k,b\s)$ be the local $G$-shtuka with $K$-structure associated with a point $y$ in the closure of $\CC_{\ul\BG,S}$ and in $\CN_{\nu,S}$. As the Newton points of $\ul\BG$ and $\ul \CF$ coincide there is a $g\in LG(k)$ whose equivalence class modulo $K$ lies in $X_{\preceq\mu,K}(b)$ with $g^{-1}b\s(g)=b_0$, and all such $g$ lie in one orbit of $J=J_b$. We may replace $S$ by the completion of $\CN_{\nu,S}$ in $y$. Let $\CN_{\nu}$ be the closed Newton stratum in the universal deformation of $\ul \CF$ bounded by $\mu$. Then there is a morphism $S\rightarrow \mathcal{N}_{\nu}$ such that the local $G$-shtuka on $S$ is the pull-back of the universal local $G$-shtuka with $K$-structure over $\CN_{\nu}$. We consider the inverse image in $X_{\preceq\mu,K}(b)^\wedge\,\whtimes_k\,(I_{\ol N}/x^{-1}I_{\ol N}x)^{\wedge}$ of the image of $\CC_{\ul\BG,S}$ in $\mathcal{N}_{\nu}$. By assumption it is nonempty. We consider its first component (in $X_{\preceq\mu,K}(b)^{\wedge}$) in the different points of this inverse image. On the one hand, the local $G$-shtuka with $K$-structure corresponding to such a point is isomorphic to $\ul\BG$, hence the first component has to be of the form $jgK$ for some $j\in J_b$. Note that $jgK\in LG/K$ is a closed point defined over $k$ for each $j$. On the other hand it is in the completion of $X_{\preceq\mu,K}(b)$ at $1$. By the same argument as in the proof of the last assertion of Theorem \ref{ThmProductStructure} we see that the first component has to be constant, and equal to the unique closed point $1$ of $X_{\preceq\mu,K}(b)^{\wedge}$. Thus $K\cap Jg\neq\emptyset$, hence $b_0$ and $b$ are $\sigma$-conjugate under $K$, and $y\in \CC_{\ul\BG,S}$.
\end{proof}

\begin{corollary}\label{appladlv}
Let $b\in LG(k)$ and let $\nu$ be the Newton point of $b$.
\begin{enumerate}
\item Let $\mu\in X_*(\T)$ be dominant with $X_{\mu}(b)\neq \emptyset$. Then $X_{\preceq\mu}(b)$ and $X_{\mu}(b)$ are equidimensional of dimension $\langle\rho,\mu-\nu\rangle-\frac{1}{2}(\rk_{\mathbb{F}_q\dpl z\dpr}(G)-\rk_{\mathbb{F}_q\dpl z\dpr}(J_b))$. In particular, $X_{\mu}(b)$ is open and dense in $X_{\preceq\mu}(b)$.
\item Let $y\in \wt W$ with $X_y(b)\neq \emptyset$. Let $C$ be an irreducible component of $X_y(b)$ and let $\ell(\nu,y,C)$ denote the maximal length of a chain of Newton points $\succeq\nu$ which is realizable in the universal deformation of Weyl type $y$ of a point $\ul\CG$ associated with $g^{-1}b\s(g)$ for some $g$ in the smooth locus of $C$ (see \cite[Definition 7.6]{HV1}). Then $\dim C\geq\ell(y)-\langle 2\rho,\nu\rangle-\ell(\nu, y)$. On the other hand $\dim X_y(b)=\ell(y)-\langle 2\rho,\nu\rangle-\codim(\CN_{\nu})$ where $\codim(\CN_{\nu})$ denotes the codimension of the Newton stratum associated with $\nu$ in $IyI$.
\end{enumerate}
\end{corollary}

\begin{remark}
For the affine Deligne-Lusztig varieties $X_{\mu}(b)$ it is known by \cite{GHKR}, \cite{Viehmann06} that their dimension is equal to $\langle\rho,\mu-\nu\rangle-\frac{1}{2}(\rk_{\mathbb{F}_q\dpl z\dpr}(G)-\rk_{\mathbb{F}_q\dpl z\dpr}(J_b))$. In \cite{HV1} we proved both assertions of the corollary under the additional assumption that $b$ is basic. Equidimensionality of $X_{\mu}(b)$ for $b\in\T(\mathbb{F}_q\dpl z\dpr)$ (not necessarily basic) is shown in \cite[Proposition 2.17.1]{GHKR}. The last assertion of the corollary proves \cite[Conjecture 1]{Beazley}. In general, the affine Deligne-Lusztig varieties $X_y(b)$ are not equidimensional (see \cite{GoertzHe}).
\end{remark}

\begin{proof}[Proof of Corollary \ref{appladlv}]
Let $X$ be an irreducible component of the affine Deligne-Lusztig variety. Let $g$ be a smooth point on $X$. By replacing $b$ by $g^{-1}b\s(g)$ we may assume that $g=1$ is smooth and lies on the component we want to consider. Let $\ul\CG$ be the local $G$-shtuka with $K_0$- respectively $I$-structure associated with $b$. We consider its universal deformation by local $G$-shtukas bounded by $\mu$, respectively of affine Weyl type $y$, which by Proposition~\ref{PropDefo} is pro-representable by a complete noetherian local ring $\ol\CalD$ of dimension $\langle 2\rho,\mu\rangle$, respectively $\ell(y)$. By \cite[Corollary 7.7]{HV1}, the codimension of an irreducible component of the Newton stratum associated with the Newton point $\nu$ of $\ul\CG$ in $\Spec \ol\CalD$ is less or equal to the maximal length of a chain of Newton points ordered by $\preceq$ and lying between $\mu$ and $\nu$, respectively realizable in $(IyI)^{\wedge}_b$ (in the second case this is what we call $\ell(\nu,y)$ above). Thus in the first case we obtain as in \cite[Proposition 7.8]{HV1} that the dimension of each irreducible component of the Newton stratum is at least $\langle\rho,\mu+\nu\rangle-\frac{1}{2}(\rk_{\mathbb{F}_q\dpl z\dpr}(G)-\rk_{\mathbb{F}_q\dpl z\dpr}(J_b))$. The assertions of the corollary now follow from the product structures in Theorem \ref{ThmProductStructure} and Theorem \ref{ThmProductI} together with the fact that $\dim (I_{\ol N}/x^{-1}I_{\ol N}x)=\dim (I/(I\cap x^{-1}Ix))=\ell(x)=\langle 2\rho,\nu\rangle$ by \cite[Proposition 13.1.3 (3)]{GHKR2}. Here $x$ is a $P$-fundamental alcove corresponding to $[b]$. 
\end{proof}

\begin{corollary}
Let $\ul\BG$ be a local $G$-shtuka with $K$-structure over $k$ bounded by some $\mu\in X_\ast(T)$. 
\begin{enumerate}
\item Let $\nu$ be the Newton polygon of $\ul\BG$. Then the closed Newton stratum in the universal deformation ring for deformations of $\ul\BG$ that are bounded by $\mu$ is equidimensional of dimension $\langle\rho,\mu+\nu\rangle-\frac{1}{2}(\rk_{\mathbb{F}_q\dpl z\dpr}(G)-\rk_{\mathbb{F}_q\dpl z\dpr}(J_b))$.
\item The generic fiber of the universal local $G$-shtuka with $K$-structure over the universal deformation ring for deformations of $\ul\BG$ that are bounded by $\mu$ has Newton polygon $\mu$.
\end{enumerate}
\end{corollary}
\begin{proof}
The first assertion follows from Corollary \ref{appladlv} and Theorem \ref{ThmProductStructure}. Let $\nu_0$ be the Newton point at the generic point of the deformation space.  By \cite[Corollary 7.7(b)]{HV1} we obtain $$\sum_i \lceil\langle\omega_i,\mu-\nu\rangle\rceil\leq \sum_i \lceil\langle\omega_i,\nu_0-\nu\rangle\rceil.$$ The two sides of this inequality are the maximal lengths of chains of comparable Newton polygons between $\nu$ and $\mu$, respectively $\nu_0$ (see \cite[Theorem 7.4 (iv)]{Chai}). As $\nu_0\preceq\mu$, the right hand side is strictly smaller than the left hand side unless $\mu=\nu_0$.
\end{proof}

%%%%%%%%%%%%%%%%%%%%%%%%%%%%%%%%%%%%%%%%%%%%%%%%%%%%%%%%%%%%%%%%%%%%%%
%
%    Bibliography
%
%%%%%%%%%%%%%%%%%%%%%%%%%%%%%%%%%%%%%%%%%%%%%%%%%%%%%%%%%%%%%%%%%%%%%%

\vfill

\begin{minipage}[t]{0.5\linewidth}
\noindent
Urs Hartl\\
Universit\"at M\"unster\\
Mathematisches Institut \\
Einsteinstr.~62\\
D -- 48149 M\"unster
\\ Germany
\\[1mm]
{\small www.math.uni-muenster.de/u/urs.hartl/}
\end{minipage}
\begin{minipage}[t]{0.45\linewidth}
\noindent
Eva Viehmann\\
Universit\"at Bonn\\
Mathematisches Institut \\
Endenicher Allee 60\\
D -- 53115 Bonn
\\ Germany
\\[1mm]
\end{minipage}

\end{document}